%% file: SHT_els.tex
\documentclass[final,3p,times]{elsarticle}

\input{header_els}

\graphicspath{{./Codes/}}

\usepackage{rotfig,rotfig_tikz}

\journal{Applied and Computational Harmonic Analysis}

\begin{document}

\begin{frontmatter}



\title{Fast and backward stable transforms between spherical harmonic expansions and bivariate Fourier series}


\author[RMS]{Richard Mika\"el Slevinsky\corref{cor1}}
\ead{Richard.Slevinsky@umanitoba.ca}
\ead[url]{https://home.cc.umanitoba.ca/~slevinrm/}

\cortext[cor1]{Corresponding author}
\address[RMS]{Department of Mathematics, University of Manitoba, Winnipeg, Manitoba, Canada}

\begin{abstract}
A rapid transformation is derived between spherical harmonic expansions and their analogues in a bivariate Fourier series. The change of basis is described in two steps: firstly, expansions in normalized associated Legendre functions of all orders are converted to those of order zero and one; then, these intermediate expressions are re-expanded in trigonometric form. The first step proceeds with a butterfly factorization of the well-conditioned matrices of connection coefficients. The second step proceeds with fast orthogonal polynomial transforms via hierarchically off-diagonal low-rank matrix decompositions. Total pre-computation requires at best $\mathcal{O}(n^3\log n)$ flops; and, asymptotically optimal execution time of $\mathcal{O}(n^2\log^2 n)$ is rigorously proved via connection to Fourier integral operators.
\end{abstract}


\end{frontmatter}


\section{Introduction}

Spherical harmonics are the natural basis for square integrable functions on the sphere in the sense of Lebesgue. As eigenfunctions of the Laplace--Beltrami operator, spherical harmonics are mathematically elegant; and with orthonormalization, stable numerical methods are available for tensor calculus. Another basis for the sphere is derived by doubling the co-latitudinal angle and extending the original function (or data) with block-mirror centrosymmetry. A bivariate Fourier basis results, and by the grace of the fast Fourier transform (FFT) comes rapid evaluation at points equispaced-in-angle. Through nonuniform FFTs, rapid evaluation is extended to bivariate grids with little to no regard for their distribution~\cite{Dutt-Rokhlin-14-1368-93,Potts-Steidl-24-2013-03,Greengard-Lee-46-443-04,Ruiz-Antolin-Townsend-17}. Townsend, Wilber, and Wright~\cite{Townsend-Wilber-Wright-38-C403-16} have recently experimented with low-rank compression techniques for bivariate Fourier series on the sphere.

Several numerical methods exist for the purposes of the synthesis and analysis of spherical harmonic expansions; that is, the rapid evaluation of spherical harmonic expansions at points (nearly) equispaced-in-angle and the determination of expansion coefficients by data at such points. Healy et al.~\cite{Healy-Rockmore-Kostelec-Moore-9-341-03} derive the first modern approach based on the so-called split-Legendre functions and the Driscoll--Healy sampling theorem~\cite{Driscoll-Healy-15-202-94}. Suda and Takami~\cite{Suda-Takami-71-703-02} and Kunis and Potts~\cite{Kunis-Potts-161-75-03} accelerate the implied polynomial interpolation by using the Fast Multipole Method (FMM). On the basis of classical WKB asymptotics, Mohlenkamp~\cite{Mohlenkamp-5-159-99} accelerates synthesis and analysis by constructing localized numerical approximations that incorporate phase information. More recently, Rokhlin and Tygert~\cite{Rokhlin-Tygert-27-1903-06} combine the theory of spectral connection matrices of associated Legendre functions with Chandrasekaran and Gu's fast eigensolver for symmetric block-diagonal plus semiseparable matrices~\cite{Chandrasekaran-Gu-96-723-04} to derive the first fast algorithms with asymptotically optimal pre-computation, though the optimal complexity is predicted ``only for absurdly large degrees.'' Tygert completes his trilogy with significant improvements to the first methodology~\cite{Tygert-227-4260-08} and a new method~\cite{Tygert-229-6181-10} based on the butterfly algorithm~\cite{Michielssen-Boag-44-1086-96,ONeil-Woolfe-Rokhlin-28-203-10}. Tygert's method has been implemented in several real-world applications~\cite{Reuter-Ratner-Seideman-131-094108-1-09,Seljebotn-199-12-12,Wedi-Hamrud-Mozdzynski-141-3450-13}. Thus for a new contribution, it is as important to maximize the practicality of the spherical harmonic basis as it is to completely eliminate pre-computation.

Shortly after the turn of the $20^{\rm th}$ century, Schuster published foundational work on the Fourier coefficients of spherical harmonics~\cite{Schuster-200-181-03}. In the computer era, Hofsommer and Potters~\cite{Hofsommer-Potters-63-460-60} show how one may implement the conversion in practice, and the Fourier coefficients have been more recently studied as well~\cite{Sneeuw-Bun-70-224-96,Gruber-Abrykosov-90-525-16}. Although numerical evidence is presented to suggest recursion-through-order is stable in the downward direction, extended precision is still required and the $\OO(n^3)$ complexity of the conversion is not asymptotically optimal. Conversion from spherical harmonics to bivariate Fourier series is not one-to-one. In order to make sense of inversion, a weighted least-squares solution is proposed in~\cite{Hofsommer-Potters-63-460-60}. However, when the problem is properly subdivided, a more natural least-squares problem awaits.

The present work may be viewed as an extension of the butterfly algorithm to the connection problem between spherical harmonic expansions and their bivariate Fourier series. A change of basis implies several advantages over the full synthesis and analysis. Firstly, certain subproblems are well-conditioned and backward stable algorithms are derived for their numerical solution; this obviates the necessity for extended precision~\cite{Tygert-227-4260-08,Tygert-229-6181-10}, as Gaussian quadratures are not part of the algorithm. Secondly, in the light of the connection problem, a so-called {\em rank property} of the matrices of connection coefficients is rigorously proved, establishing the asymptotically optimal complexity $\OO(n^2\log^2 n)$. This rank property is also universal in its proof of Tygert's acceleration of synthesis and analysis as well~\cite{Tygert-229-6181-10}. Lastly, the formulation as a connection problem presents the possibility to improve the pre-computation by requiring a mere skeleton of the full plan as fast and backward stable algorithms may convert between expansions in neighbouring orders. The main technological asset of the butterfly algorithm is the interpolative decomposition; algorithms for interpolative decompositions include~\cite{Cheng-Gimbutas-Martinsson-Rokhlin-26-1389-05,Liberty-et-al-104-20167-07}. An open source implementation is available in the software package {\tt FastTransforms.jl}~\cite{Slevinsky-GitHub-FastTransforms}. 

Two prominent numerical methods for evolving stiff time-dependent partial differential equations are exponential integrators such as the fourth-order Runge--Kutta scheme ETDRK4~\cite{Cox-Matthews-176-430-02}, and implicit-explicit (IMEX) Runge--Kutta methods such as those developed in~\cite{Ascher-Ruuth-Spiteri-25-151-97}. While IMEX schemes are versatile when considering more general stiff linear operators, their order may drop to one in the stiff limit. On the other hand, exponential integrators deal with the stiffness exactly, but an efficient implementation requires a basis that diagonalizes the stiff linear operator. In the basis of spherical harmonics, certain linear differential operators such as the Laplace--Beltrami operator are diagonal. Similarly, common nonlinearities are diagonalized by pointwise evaluation on a grid. Thus a method that utilizes both spherical harmonics and bivariate Fourier series and may stably convert between representations is ideal for time evolution.

The conversion of expansions in associated Legendre functions to those involving zeroth and first orders is useful for the solution of partial differential equations with initial conditions in the spherical harmonic basis and whose partial differential operators on the sphere have low splitting rank. On the sphere, this offers the advantage of resolution in an isometric Hilbert space. A global spectral method for such problems is derived by Townsend and Olver~\cite{Townsend-Olver-299-106-15}.

That spherical harmonics diagonalize certain singular integral operators on the surface of the sphere could be the basis for extending the ultraspherical spectral method for singular integral equations~\cite{Slevinsky-Olver-332-290-17}, which results in highly-structured linear systems, to multi-sphere scattering with spectral accuracy.

\section{Fundamentals}

Let $\mu$ be a positive Borel measure on $D\subset\R^n$. The inner product:
\begin{equation}
\langle f, g \rangle = \int_D \conj{f(x)} g(x)\ud\mu(x),
\end{equation}
where $\conj{f(x)}$ denotes complex conjugation, induces the norm $\norm{f}_2 = \sqrt{\langle f,f\rangle}$ and the associated Hilbert space $L^2(D,\ud\mu(x))$. In case of ambiguity, the notation $\langle f, g\rangle_{{\rm d}\mu}$ is used to distinguish between different measures.

Let $\Sph^2\subset\R^3$ denote the unit $2$-sphere, $\theta\in[0,\pi]$ the co-latitudinal angle, $\varphi\in[0,2\pi)$ the longitudinal angle, and $\ud\Omega = \sin\theta\ud\theta\ud\varphi$ the measure generated by the solid angle $\Omega$ subtended by a spherical cap. Then, any function $f\in L^2(\Sph^2,\ud\Omega)$ may be expanded in spherical harmonics:
\begin{equation}\label{eq:sphericalharmonicexpansion}
f(\theta,\varphi) = \sum_{\ell=0}^{+\infty}\sum_{m=-\ell}^{+\ell} f_\ell^m Y_\ell^m(\theta,\varphi) = \sum_{m=-\infty}^{+\infty}\sum_{\ell=\abs{m}}^{+\infty} f_\ell^m Y_\ell^m(\theta,\varphi),
\end{equation}
where the expansion coefficients are:
\begin{equation}
f_\ell^m = \dfrac{\langle Y_\ell^m, f\rangle}{\langle Y_\ell^m, Y_\ell^m\rangle}.
\end{equation}
Let $\N_0$ denote the non-negative integers. Bandlimiting Eq.~\eqref{eq:sphericalharmonicexpansion} to $\ell\le n\in\N_0$ results in the best degree-$n$ trigonometric polynomial approximation of $f\in L^2(\Sph^2,\ud\Omega)$.

Using the Condon--Shortley phase convention~\cite{Condon-Shortley-51}, orthonormal spherical harmonics are given by:
\begin{equation}\label{eq:sphericalharmonics}
Y_\ell^m(\theta,\varphi) = \dfrac{e^{\ii m\varphi}}{\sqrt{2\pi}} \underbrace{\ii^{m+|m|}\sqrt{(\ell+\tfrac{1}{2})\dfrac{(\ell-m)!}{(\ell+m)!}} P_\ell^m(\cos\theta)}_{\tilde{P}_\ell^m(\cos\theta)},\qquad \ell\in\N_0,\quad -\ell\le m\le \ell,
\end{equation}
and we will find the connection between associated Legendre functions and ultraspherical polynomials~\cite[\S 18.11.1]{Olver-et-al-NIST-10} particularly useful:
\begin{equation}\label{eq:ALPtoUS}
P_\ell^m(\cos\theta) = (-2)^m(\tfrac{1}{2})_m(\sin\theta)^m C_{\ell-m}^{(m+\frac{1}{2})}(\cos\theta).
\end{equation}
In Eq.~\eqref{eq:sphericalharmonics}, the notation $\tilde{P}_\ell^m$ is used to denote orthonormality for fixed $m$ in the sense of $L^2([-1,1],\ud x)$. In Eq.~\eqref{eq:ALPtoUS}, $(x)_n = \frac{\Gamma(x+n)}{\Gamma(x)}$ is the Pochhammer symbol~\cite{Abramowitz-Stegun-65}.

\subsection{The spherical harmonic connection problem}

Expansions in families of orthogonal functions may be related by a change of basis. This defines the classical connection problem.

\begin{definition}
Let $\{\phi_n(x)\}_{n\in\N_0}$ be a family of orthogonal functions with respect to $L^2(\hat{D},\ud\hat{\mu}(x))$ and let $\{\psi_n(x)\}_{n\in\N_0}$ be another family of orthogonal functions with respect to $L^2(D,\ud\mu(x))$. The connection coefficients:
\begin{equation}
c_{\ell,n} = \dfrac{\langle \psi_\ell, \phi_n\rangle_{{\rm d}\mu}}{\langle \psi_\ell, \psi_\ell\rangle_{{\rm d}\mu}},
\end{equation}
allow for the expansion:
\begin{equation}
\phi_n(x) = \sum_{\ell=0}^\infty c_{\ell,n} \psi_\ell(x).
\end{equation}
\end{definition}

\begin{theorem}
Let $\{\phi_n(x)\}_{n\in\N_0}$ and $\{\psi_n(x)\}_{n\in\N_0}$ be two families of {\em orthonormal} functions with respect to $L^2(D,\ud\mu(x))$. Then the connection coefficients satisfy:
\begin{equation}
\sum_{\ell=0}^\infty \conj{c_{\ell,m}}c_{\ell,n} = \delta_{m,n},
\end{equation}
where $\delta_{m,n}$ is the Kronecker delta function~\cite[Chap. 24]{Abramowitz-Stegun-65}.
\end{theorem}
\begin{proof}
We expand the orthonormal functions $\phi_m$ and $\phi_n$ in the basis $\{\psi_\ell\}_{\ell\in\N_0}$, and integrate the product:
\begin{equation}
\delta_{m,n} = \langle \phi_m, \phi_n\rangle = \int_D \conj{\left(\sum_{\ell=0}^\infty c_{\ell,m} \psi_\ell(x)\right)}\left(\sum_{\ell=0}^\infty c_{\ell,n} \psi_\ell(x)\right)\ud\mu(x) = \sum_{\ell=0}^\infty \conj{c_{\ell,m}}c_{\ell,n}.
\end{equation}
\end{proof}

Since the space of algebraic polynomials of degree at most $n$ is finite-dimensional, if the source and target families of orthogonal functions consist of a weight function multiplying algebraic polynomials, then the expansion of one family in the other basis is finite-dimensional and the connection coefficients populate a matrix. If, in addition, the source and target families are orthonormal, the matrix of connection coefficients has orthonormal columns and thereby unit singular values. Conversion of expansions between families of orthonormal functions in the same finite-dimensional Hilbert space is therefore well-conditioned in the sense of the induced matrix $2$-norm.

Any matrix $A\in\R^{m\times n}$, $m\ge n$, with orthonormal columns has the Moore--Penrose pseudoinverse $A^+ = A^\top$. Where the classical connection problem is finite-dimensional but rectangular, the Moore--Penrose pseudoinverse provides a least-squares solution to the inversion problem.

While much is gained by discussing connection coefficients in such generality, we will now consider the connection coefficients pertinent to the spherical harmonics. The normalized associated Legendre functions are a family of orthonormal functions for the same Hilbert space $L^2([-1,1],\ud x)$. In fact, all even-ordered $\tilde{P}_\ell^m(x)$ are polynomials and all odd-ordered $\tilde{P}_\ell^m(x)$ are polynomials multiplied by the weight $\sqrt{1-x^2}$.

The three-term recurrence relations for the associated Legendre functions may be combined to produce a five-term recurrence relation satisfied by the connection coefficients~\cite{Salzer-16-705-73}. However, it is theoretically challenging to prove backward stability of this formulation even though the problem is well-conditioned. Furthermore, the formulation via recurrence relations either requires $\OO(n^3)$ square roots, $\OO(n^3)$ flops, and $\OO(1)$ storage for on-the-fly calculation or $\OO(n^3)$ flops and $\OO(n^3)$ storage if the matrices are pre-computed. Either case presents us with a significant compromise. Instead, we will derive a {\em backward stable} method that either requires $\OO(n^2)$ square roots, $\OO(n^3)$ flops, and $\OO(1)$ storage for on-the-fly calculation or $\OO(n^3)$ flops and $\OO(n^2)$ storage if the representation is pre-computed.

\begin{definition}
Let $G_n$ denote the real Givens rotation:
\[
G_n = \begin{bmatrix}
1 & \cdots & 0 & 0 & 0 & \cdots & 0\\
\vdots & \ddots & \vdots & \vdots & \vdots & & \vdots\\
0 & \cdots & c_n & 0 & s_n & \cdots & 0\\
0& \cdots & 0 & 1 & 0 & \cdots & 0\\
0 & \cdots & -s_n & 0 & c_n & \cdots & 0\\
\vdots & & \vdots & \vdots & \vdots & \ddots & \vdots\\
0 & \cdots & 0 & 0 & 0 & \cdots & 1\\
\end{bmatrix},
\]
where the sines $s_n = \sin\theta_n$ and the cosines $c_n = \cos\theta_n$, for some $\theta_n\in[0,2\pi)$, are in the intersections of the $n^{\rm th}$ and $n+2^{\rm nd}$ rows and columns, embedded in the identity of a conformable size. 
\end{definition}

Let $I_{m\times n}$ denote the rectangular identity matrix with ones on the main diagonal and zeros everywhere else.

\begin{theorem}\label{theorem:SS}
The connection coefficients between $\tilde{P}_{n+m+2}^{m+2}(\cos\theta)$ and $\tilde{P}_{\ell+m}^m(\cos\theta)$ are:
\begin{equation}\label{eq:SScoefficients}
c_{\ell,n}^{m} = \left\{\begin{array}{ccc} (2\ell+2m+1)(2m+2)\sqrt{\dfrac{(\ell+2m)!}{(\ell+m+\frac{1}{2})\ell!}\dfrac{(n+m+\frac{5}{2})n!}{(n+2m+4)!}}, & \for & \ell \le n,\quad \ell+n\hbox{ even},\\
-\sqrt{\dfrac{(n+1)(n+2)}{(n+2m+3)(n+2m+4)}}, & \for & \ell = n+2,\\
0, & & otherwise.
\end{array} \right.
\end{equation}
Furthermore, the matrix of connection coefficients $C^{(m)} \in \R^{(n+3)\times (n+1)}$ may be represented via the product of $n$ Givens rotations:
\[
C^{(m)} = G_0^{(m)}G_1^{(m)}\cdots G_{n-2}^{(m)}G_{n-1}^{(m)} I_{(n+3)\times (n+1)},
\]
where the sines and cosines for the Givens rotations are given by:
\begin{equation}\label{eq:GRcoefficients}
s_n^m = \sqrt{\dfrac{(n+1)(n+2)}{(n+2m+3)(n+2m+4)}},\quad{\rm and}\quad c_n^m = \sqrt{\dfrac{(2m+2)(2n+2m+5)}{(n+2m+3)(n+2m+4)}}.
\end{equation}
\end{theorem}
\begin{proof}
Both recurrence relations involving ultraspherical polynomials~\cite[\S 18.9.7 \& \S 18.9.8]{Olver-et-al-NIST-10} are essentially:
\begin{align*}
C_n^{(m+\frac{1}{2})}(\cos\theta) & = \dfrac{2m+1}{2n+2m+1}\left( C_n^{(m+\frac{3}{2})}(\cos\theta) - C_{n-2}^{(m+\frac{3}{2})}(\cos\theta)\right),\quad{\rm and}\\
\sin^2\theta C_n^{(m+\frac{5}{2})}(\cos\theta) & = \dfrac{(n+2m+3)(n+2m+4)}{(2m+3)(2n+2m+5)}C_n^{(m+\frac{3}{2})}(\cos\theta) - \dfrac{(n+1)(n+2)}{(2m+3)(2n+2m+5)}C_{n+2}^{(m+\frac{3}{2})}(\cos\theta).
\end{align*}
Thus, the entire transformation from an expansion in $\sin^2\theta C_n^{(m+\frac{5}{2})}(\cos\theta)$ to an expansion in $C_n^{(m+\frac{1}{2})}(\cos\theta)$ may be represented by the product of the inverse of a square upper-triangular banded matrix and a rectangular banded matrix both with bandwidth $2$. The inverse of an upper-triangular banded matrix of bandwidth $2$ is a diagonal-plus-upper-semiseparable matrix of semiseparability rank $2$. Translating this result into the notation of normalized associated Legendre functions, we arrive at $c_{\ell,n}^m$ in Eq.~\eqref{eq:SScoefficients}.

The parity in the source and target bases encodes a chessboard pattern of zeros in $C^{(m)}$. Start by applying a Givens rotation from the left to introduce a zero in the third row of the first column. Since the columns of $C^{(m)}$ are orthonormal, $s_0^m = -c_{2,0}^m$. Apply another Givens rotation from the left to introduce a zero in the fourth row of the second column of the conversion matrix. Again, we find that $s_1^m = -c_{3,1}^m$. Due to the orthonormality, the first rotation introduces zeros in every entry of the first row but the first. Similarly, the second rotation introduces zeros in every entry of the second row but the second. Continuing with $n-2$ more Givens rotations, we arrive at $I_{(n+3)\times(n+1)}$.
\end{proof}

To illustrate the order of operations, the matrices of connection coefficients may be represented as:
\[
\begin{tikzpicture}[baseline={(current bounding box.center)},scale=1.6,y=-1cm]
  \normalizedfusedconversion{0.0}{0.0}{$C^{(m)}$};
  \node (equals) at (1.75,0.5) {\Large $=$};
  \foreach \j in {0.0,0.1,0.2} {
    \tikzrotation{2.0*\j+2.4}{\j}
  }
  \foreach \j in {0.6,0.7,0.8} {
    \tikzrotation{2.0*\j+2.2}{\j}
  }
  \node (ddots) at (3.1,0.5) {$\ddots$};
  \drawbrace{2.4}{3.8}{1.1}{$G_0^{(m)}\cdots G_{n-1}^{(m)}$}
  \rectangular{4.0}{0.0}{{\relscale{0.9}$I_{(n+3)\times(n+1)}$}};
\end{tikzpicture}
\]
The arrows indicate which rows are altered by a Givens rotation; rotations nearest the rectangular identity are applied first.

For a real number $x$, we denote floating-point approximations to quantities by $\fl(x)$. Using the IEEE 754-2008 standard model for floating-point arithmetic, whenever $x$ and $y$ are floating-point numbers and $\circledast$ is one of the four operations $+$, $-$, $\times$, or $\div$, we have:
\begin{equation}
\fl(x\circledast y) = (x\circledast y)(1+\delta)^{\pm1},\quad{\rm where}\quad |\delta| \le \epsilon_{\rm mach},
\end{equation}
where $\epsilon_{\rm mach}$ is a unit of least precision (ulp). As well, whenever $x$ is a floating-point number:
\begin{equation}
\fl(\!\sqrt{x}) = \sqrt{x}\,(1+\delta)^{\pm1}\quad{\rm where}\quad |\delta| \le \epsilon_{\rm mach}.
\end{equation}

Since the sines and the cosines for the Givens rotations are derived analytically, rather than one numerically from the other through the relationship $s^2+c^2=1$, the rounding errors committed on a computer with finite-precision arithmetic are on the order of $\epsilon_{\rm mach}$ multiplied by the relative sizes of the components {\em independently}. The precise implementation is now discussed.

\begin{algorithm}\label{alg:FPE}
The computation of $s_n^m$ and $c_n^m$ is executed as follows:
\begin{enumerate}
\item Numerators and denominators are calculated in IEEE $64$-bit signed integer arithmetic. Since $m\le n$, the largest integers encountered in $s_n^m$ or $c_n^m$ are:
\[
(3n+3)(3n+4) = 9n^2+21n+12,
\]
in the denominators. The denominators, and consequently numerators, are exactly computed so long as $n \le 1,012,333,498$ (computed by solving for the largest integer satisfying $9n^2+21n+12 \le 2^{63}-1$);
\item Exactly computed numerators and denominators are converted to IEEE $64$-bit floating-point numbers. This conversion is lossless so long as $n \le 31,635,420$ (computed by solving for the largest integer satisfying $9n^2+21n+12 \le 2^{53}$, the largest representable integer by the floating-point number's significand);
\item For each sine and cosine, at most one rounding error bounded by one ulp is committed in dividing numerator by denominator; and,
\item For each sine and cosine, at most another rounding error bounded by one ulp is committed in computing the square root.
\end{enumerate}
\end{algorithm}
Due to memory restrictions, a bandlimit of $n = 31,635,420$ is unlikely to be surpassed in our lifetime\footnote{In case it were surpassed, then the algorithm above could easily be modified to compute individual ratios in $s_n^m$ and $c_n^m$ in $64$-bit floating-point rather than computing the full numerators and denominators. This modification would still be backward stable but involves a few more sources of rounding errors.}. Furthermore, there is neither underflow nor overflow. This ensures that an algorithm designed on the application of these Givens rotations is indeed $2$-normwise backward stable since component-wise bounds are tighter than bounds on the induced matrix $2$-norm. The result of Algorithm~\ref{alg:FPE} is that $\fl(G_n^{(m)}) = G_n^{(m)} + E$, where $\norm{E}_2 = \OO(\epsilon_{\rm mach})$.

Algorithm~\ref{alg:FPE} is also a backward stable implementation of left inversion (transposition) of matrices of connection coefficients by transposition of the product of Givens rotations. In fact, the Givens rotations correspond to a structured $QR$ factorization of $C^{(m)}$ and thus their use to solve a least-squares problem enjoys $2$-normwise relative backward stability since the condition number of $C^{(m)}$ is $1$. Following the text after Higham~\cite[Theorem 20.3]{Higham-02}, ``The conclusion is that, unless $A$ is very ill-conditioned, the residual $b - A\hat{x}$, [computed by a $2$-normwise relative backward stable $QR$ factorization of $A$], will not exceed the larger of the true residual $r = b - Ax$ and a constant multiple of the error in evaluating $\fl(r)$-a very satisfactory result.''

\subsection{The interpolative decomposition}

The interpolative decomposition (ID) of a rectangular matrix $A\in\R^{m\times n}$ factorizes the matrix into one whose $k$ columns constitute a subset of unique columns of $A$, and another containing the $k$-by-$k$ identity matrix as a subset and whose remaining entries are not too large. The following lemma justifies the approximation power of an ID for low-rank matrices.

\begin{lemma}~\cite{Cheng-Gimbutas-Martinsson-Rokhlin-26-1389-05,Liberty-et-al-104-20167-07,Tygert-229-6181-10}
Consider a rectangular matrix $A\in\R^{m\times n}$. Then, for any $k\in\N_0$ with $k\le \min\{m,n\}$, there exist $A_{\rm CS}\in\R^{m\times k}$ whose columns constitute a subset of the columns of $A$ and $A_{\rm I}\in\R^{k\times n}$ such that:
\begin{enumerate}
\item some subset of the columns of $A_{\rm I}$ makes up the $k\times k$ identity matrix;
\item $\max_{1\le i\le k,1\le j\le n} \abs{(A_{\rm I})_{i,j}} \le 2$;
\item the spectral norm of $A_{\rm I}$ satisfies $\norm{A_{\rm I}}_2 \le \sqrt{4k(n-k)+1}$;
\item the least singular value of $A_{\rm I}$ is at least $1$;
\item $A_{\rm CS} A_{\rm I} = A$ whenever $k = m$ or $k = n$; and,
\item when $k<\min\{m,n\}$, the spectral norm of $A - A_{\rm CS} A_{\rm I}$ satisfies:
\[
\norm{A - A_{\rm CS} A_{\rm I}}_2 \le \sqrt{4k(n-k)+1}\sigma_{k+1},
\]
where $\sigma_{k+1}$ is the $k+1^{\rm st}$ singular value of $A$.
\end{enumerate}
\end{lemma}

The subscript $_{\rm CS}$ stands for {\em column skeleton} and $_{\rm I}$ stands for {\em interpolation}.

This lemma states that $A \approx A_{\rm CS}A_{\rm I}$ provided that $\sigma_{k+1}$ is sufficiently small. For any $\varepsilon>0$, numerous disparate algorithms~\cite{Gu-Eisenstat-17-848-96,Cheng-Gimbutas-Martinsson-Rokhlin-26-1389-05,Liberty-et-al-104-20167-07} can identify the least $k$ such that $\norm{A - A_{\rm CS} A_{\rm I}}_2 \approx \varepsilon$. The algorithms compute both $A_{\rm CS}$ and $A_{\rm I}$ using at most $\OO(kmn\log n)$ operations, typically requiring only $\OO(kmn)$ operations. Furthermore, there is abundant empirical evidence that the ID is a numerically stable algorithm for computing (transposed) matrix-vector products.

\subsection{Threefold symmetry}

For improved performance, the threefold symmetry in the spherical harmonics is leveraged: negative-ordered layers are treated similarly to positive-ordered layers; even-ordered layers are necessarily treated differently than odd-ordered layers; and, the parity in the bases allows one to perfectly shuffle matrices of connection coefficients to remove the necessity to include storage for the chessboard pattern of zeros. This also halves the semiseparability rank.

Throughout, complexity results are stated for the transformation of all layers of the spherical harmonic transform in terms of the bandlimit $n$. Furthermore, some technical results are proved for even-ordered layers and omitted for odd-ordered layers in order to not overburden the reader and the exposition.

\section{Rapid algorithms for the conversion between bandlimited spherical harmonic expansions and their bivariate Fourier series}

The steps required by the algorithm are illustrated in Figure~\ref{fig:SHT}. At first, we derive a new algorithm that converts higher-order layers of the spherical harmonics into expansions with orders zero and one. Then, these coefficients are rapidly transformed into their Fourier coefficients.

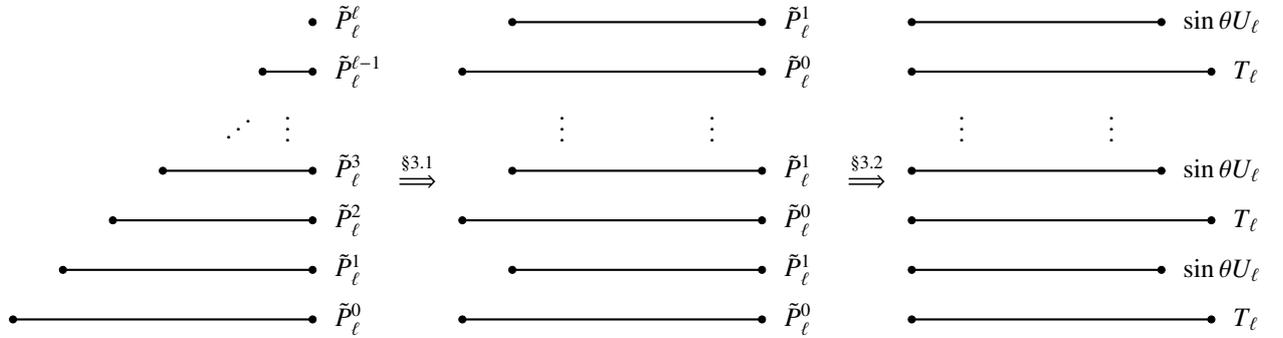
\begin{figure}[htbp]
\begin{center}
\begin{tikzpicture}[scale=0.657]
\draw[black, thick]
    (0,0) -- (6,0)
    (1,1) -- (6,1)
    (2,2) -- (6,2)
    (3,3) -- (6,3)
    (5,5) -- (6,5)
    ;
\filldraw[black]
    (0,0) circle (2pt)
    (6,0) circle (2pt)
    (1,1) circle (2pt)
    (6,1) circle (2pt)
    (2,2) circle (2pt)
    (6,2) circle (2pt)
    (3,3) circle (2pt)
    (6,3) circle (2pt)
    (5,5) circle (2pt)
    (6,5) circle (2pt)
    (6,6) circle (2pt)
    ;
\node (ellipsis1) at (4.5,4) {$\iddots$};
\node (ellipsis2) at (5.5,4) {$\vdots$};
\node[anchor=west] (zero) at (6.25,0) {$\tilde{P}_\ell^0$};
\node[anchor=west] (one) at (6.25,1) {$\tilde{P}_\ell^1$};
\node[anchor=west] (two) at (6.25,2) {$\tilde{P}_\ell^2$};
\node[anchor=west] (three) at (6.25,3) {$\tilde{P}_\ell^3$};
\node[anchor=west] (penultimate) at (6.25,5) {$\tilde{P}_\ell^{\ell-1}$};
\node[anchor=west] (ultimate) at (6.25,6) {$\tilde{P}_\ell^\ell$};
\draw[black, thick]
    (9,0) -- (15,0)
    (10,1) -- (15,1)
    (9,2) -- (15,2)
    (10,3) -- (15,3)
    (9,5) -- (15,5)
    (10,6) -- (15,6)
    ;
\filldraw[black]
    (9,0) circle (2pt)
    (15,0) circle (2pt)
    (10,1) circle (2pt)
    (15,1) circle (2pt)
    (9,2) circle (2pt)
    (15,2) circle (2pt)
    (10,3) circle (2pt)
    (15,3) circle (2pt)
    (9,5) circle (2pt)
    (15,5) circle (2pt)
    (10,6) circle (2pt)
    (15,6) circle (2pt)
    ;
\node (ellipsis1) at (11,4) {$\vdots$};
\node (ellipsis2) at (14,4) {$\vdots$};
\node[anchor=west] (zero) at (15.25,0) {$\tilde{P}_\ell^0$};
\node[anchor=west] (one) at (15.25,1) {$\tilde{P}_\ell^1$};
\node[anchor=west] (two) at (15.25,2) {$\tilde{P}_\ell^0$};
\node[anchor=west] (three) at (15.25,3) {$\tilde{P}_\ell^1$};
\node[anchor=west] (penultimate) at (15.25,5) {$\tilde{P}_\ell^0$};
\node[anchor=west] (ultimate) at (15.25,6) {$\tilde{P}_\ell^1$};
\draw[black, thick]
    (18,0) -- (24,0)
    (18,1) -- (23,1)
    (18,2) -- (24,2)
    (18,3) -- (23,3)
    (18,5) -- (24,5)
    (18,6) -- (23,6)
    ;
\filldraw[black]
    (18,0) circle (2pt)
    (24,0) circle (2pt)
    (18,1) circle (2pt)
    (23,1) circle (2pt)
    (18,2) circle (2pt)
    (24,2) circle (2pt)
    (18,3) circle (2pt)
    (23,3) circle (2pt)
    (18,5) circle (2pt)
    (24,5) circle (2pt)
    (18,6) circle (2pt)
    (23,6) circle (2pt)
    ;
\node (ellipsis1) at (19,4) {$\vdots$};
\node (ellipsis2) at (22,4) {$\vdots$};
\node[anchor=west] (zero) at (24.25,0) {$T_\ell$};
\node[anchor=west] (one) at (23.25,1) {$\sin\theta U_\ell$};
\node[anchor=west] (two) at (24.25,2) {$T_\ell$};
\node[anchor=west] (three) at (23.25,3) {$\sin\theta U_\ell$};
\node[anchor=west] (penultimate) at (24.25,5) {$T_\ell$};
\node[anchor=west] (ultimate) at (23.25,6) {$\sin\theta U_\ell$};
\node (firstarrow) at (8.1,3) {$\stackrel{\S \ref{subsection:butterfly}}{\Longrightarrow}$};
\node (secondarrow) at (17.1,3) {$\stackrel{\S \ref{subsection:FMM}}{\Longrightarrow}$};
\end{tikzpicture}
\caption{The spherical harmonic transform proceeds in two steps. Firstly, normalized associated Legendre functions are converted to normalized associated Legendre functions of order zero and one. Then, these intermediate expressions are re-expanded in trigonometric form. The first step proceeds with a butterfly factorization of the well-conditioned matrices of connection coefficients. The second step proceeds via the accelerated arithmetic of numerically approximating matrices of connection coefficients by hierarchically off-diagonal low-rank matrices.}
\label{fig:SHT}
\end{center}
\end{figure}

\subsection{The butterfly algorithm}\label{subsection:butterfly}

Suppose that $n$ is a positive integer, and $A$ is an $n\times n$ matrix. Suppose further that $\varepsilon$ and $C$ are positive real numbers, and $k$ is a positive integer, such that any contiguous rectangular subblock of $A$ containing at most $Cn$ entries can be approximated to precision $\varepsilon$ by a matrix whose rank is $k$ (using the Frobenius/Hilbert--Schmidt norm to measure the accuracy of the approximation); we will refer to this hypothesis as ``{\em the rank property}.''

The butterfly algorithm is an abstraction of the algebraic properties of divide-and-conquer fast Fourier transforms without their analytic properties. It is ideally suited for matrices resulting from the discretization of an integral operator with kernel $e^{\ii xy}$. This bivariate function is known to generate matrices which satisfy the aforementioned {\em rank property}, whereby the numerical rank of any subblock is proportional to the area in the $xy$-plane from which entries are derived.

We refer the interested reader to~\cite{Michielssen-Boag-44-1086-96,ONeil-Woolfe-Rokhlin-28-203-10} for a detailed description of the butterfly algorithm. The following synopsis assumes the aforementioned {\em rank property}, and is almost identical to the synopsis given in~\cite{Tygert-229-6181-10}. Notwithstanding, the butterfly algorithm is more generally applicable to other compressible matrices where the rank structure is unknown~\cite{Tygert-229-6181-10}, less clear, or different.

The running-time of the algorithm will be proportional to $k^2/C$; taking $C$ to be roughly proportional to $k$ suffices for many matrices of interest, so ideally $k$ should be small. We will say that two matrices $G$ and $H$ are equal to precision $\varepsilon$, denoted $G\approx H$, to mean that the spectral norm of $G-H$ is $\OO(\varepsilon)$.

We now explicitly use the rank property for subblocks of multiple heights, to illustrate the basic structure of the butterfly algorithm. Consider any two adjacent contiguous rectangular subblocks $L$ and $R$ of $A$, each containing at most $Cn$ entries and having the same number of rows, with $L$ on the left and $R$ on the right. Due to the rank property, there exists IDs:
\[
L \approx L_{\rm CS} L_{\rm I},\quad{\rm and}\quad R \approx R_{\rm CS} R_{\rm I},
\]
where the $k$ columns of $L_{\rm CS}$ are a subset of the columns of $L$ and the $k$ columns of $R_{\rm CS}$ are a subset of the columns of $R$, and where all entries in $L_{\rm I}$ and $R_{\rm I}$ have absolute value at most $2$.

Next, we merge the matrices $L$ and $R$ and split the columns of the result approximately in half, obtaining $T$ on top and $B$ on the bottom:
\[
\begin{pmatrix} L & R\end{pmatrix} = \begin{pmatrix} T\\B\end{pmatrix}.
\]
Observe that the matrices $T$ and $B$ each have at most $Cn$ entries (since $L$ and $R$ each have at most $Cn$ entries). Similarly, we merge the matrices $L_{\rm CS}$ and $R_{\rm CS}$ and split the columns of the result approximately in half, obtaining $T^{(1)}$ and $B^{(1)}$:
\[
\begin{pmatrix} L_{\rm CS} & R_{\rm CS}\end{pmatrix} = \begin{pmatrix} T^{(1)}\\B^{(1)}\end{pmatrix}.
\]
As the columns of $L_{\rm CS}$ and $R_{\rm CS}$ are columns of the original matrix, so too are the $2k$ columns of $T^{(1)}$ and $B^{(1)}$. Hence, due to the rank property, there exist IDs:
\[
T^{(1)} \approx T_{\rm CS}^{(1)}T_{\rm I}^{(1)},\quad{\rm and}\quad B^{(1)} \approx B_{\rm CS}^{(1)}B_{\rm I}^{(1)}.
\]
Combining this multilevel approximation, the top matrix $T$ and the bottom matrix $B$ are approximately:
\[
T \approx T_{\rm CS}^{(1)}T_{\rm I}^{(1)} \begin{pmatrix} L_{\rm I} & 0\\ 0 & R_{\rm I}\end{pmatrix},\quad{\rm and}\quad B \approx B_{\rm CS}^{(1)}B_{\rm I}^{(1)} \begin{pmatrix} L_{\rm I} & 0\\ 0 & R_{\rm I}\end{pmatrix}.
\]
If we use $m$ to denote the number of rows in $L$ and $R$, then the number of columns in $L$ and $R$ is at most $Cn/m$, and so the total number of entries in the matrices in the right-hand sides of the first set of IDs can be as large as $2mk+2k(Cn/m)$, whereas the total number of nonzero entries in the matrices in the right-hand sides of the bi-level representation is at most $mk+4k^2+2k(Cn/m)$. If $m$ is nearly as large as possible -- nearly $n$ -- and $k$ and $C$ are much smaller than $n$, then $mk+4k^2+2k(Cn/m)$ is about half $2mk+2k(Cn/m)$. Thus, the bi-level representation is more efficient than that provided by the single-level representation, in terms of the complexity of both storage and matrix arithmetic. This gain in efficiency is due to the rank property holding for subblocks of multiple heights.

Naturally, we may repeat this process of merging adjacent blocks and splitting in half the columns of the result, updating the compressed representations after every split. We start by partitioning $A$ into blocks each dimensioned $n\times \lfloor C\rfloor$ (except possibly for the rightmost block, which may have fewer than $\lfloor C \rfloor$ columns), and then recursively group unprocessed blocks into disjoint pairs, processing these pairs by merging and splitting them into new, unprocessed blocks having fewer rows. The resulting multilevel representation of $A$ allows us to apply $A$ and $A^\top$ with precision $\varepsilon$ using just $\OO((k^2/C)n\log n)$ operations. This is due to the fact that there are $\OO(\log n)$ levels in the representation and each level except the last will only involve $\OO(n/C)$ interpolation matrices of dimensions $k\times (2k)$.

In practice, it is crucial to use a compact form for the interpolation matrices where the embedded identity is applied as a column permutation, and to ensure that the relative tolerance of $\varepsilon \max\{m,n\}$ is used to determine the numerical rank, where $m$ and $n$ denote the dimensions of the current subblock under compression.

\subsubsection{Interpretation as Fourier integral operators}

Due to the semiseparability of the matrices of connection coefficients, the analytic rank structure of contiguous subblocks above the $2m^{\rm th}$ subdiagonal is at most $2m$; straddling this subdiagonal, the ranks of the subblocks are undetermined; and any subblock that is well below the subdiagonal has rank $0$. Using only analytic rank structure, it appears that the butterfly algorithm may be ineffective when $m$ is large. However, the analytic structure overlooks the potential for decay in the singular values of the contiguous subblocks in the semiseparable region. To elucidate this numerical rank structure, we interpret the connection coefficients as Fourier integral operators.

A rank-$1$ (bivariate) function is separable: if $f_1(x,y) = g(x)h(y)$, then $\rank(f_1) = 1$.

Similarly, a rank-$k$ (bivariate) function is a non-trivial sum of $k$ rank-$1$ functions.

\begin{definition}
A function $f\in L^p([-1,1]^2)$ is said to have rank $k_\varepsilon$ to precision $\varepsilon>0$ if:
\[
k_\varepsilon = \inf_{k\in\N_0}\left\{\inf_{\rank(f_k)\le k}\norm{f-f_k}_p \le \varepsilon \norm{f}_p\right\}.
\]
\end{definition}

\begin{definition}
Let $A: L^2([-1,1])\to L^2([-1,1])$ be the integral operator with kernel $f \in L^2([-1,1]^2)$, given by:
\[
A\{u\}(x) = \int_{-1}^1 f(x,y)u(y)\ud y.
\]
The operator $A$ is said to have rank $k_\varepsilon$ to precision $\varepsilon>0$ if $f$ has rank $k_\varepsilon$ to the same precision.
\end{definition}

We recall a lemma from~\cite{Landau-Pollak-41-1295-62} appearing in a similar form in~\cite{ONeil-Woolfe-Rokhlin-28-203-10}.

\begin{lemma}\label{lemma:LowRankFIO}
Suppose that $\delta$, $\varepsilon$, and $\gamma$ are positive real numbers and $\varepsilon<1$. Suppose further that the operator $F: L^2([-1,1])\to L^2([-1,1])$ is given by the formula:
\[
F\{h\}(x) = \int_{-1}^1 e^{\ii\gamma xy}h(y)\ud y.
\]
Then $F$ has rank to precision $\varepsilon$ at most:
\[
k(= k(\epsilon,\delta,\gamma)) = (1+\delta)\left(\dfrac{2\gamma}{\pi} + \dfrac{E}{\delta}\right) + 3,
\]
where:
\[
E(=E(\epsilon,\delta)) = 2\sqrt{2\ln\left(\dfrac{4}{\varepsilon}\right)\ln\left(\dfrac{6\sqrt{1/\sqrt{\delta}+\sqrt{\delta}}}{\varepsilon}\right)}.
\]
\end{lemma}

This result means that if the Fourier transform with kernel $e^{\ii kx}$ is restricted to a rectangle in the $kx$-plane, then its rank is bounded by a constant times the area of the rectangle.

For $m\in\N$, the connection coefficients between $\tilde{P}_{\ell+2m}^{2m}$ and $\tilde{P}_n^0$ are given by the inner product:
\begin{equation}\label{eq:intrepcoefficients}
c_{\ell,n}^{2m} = \int_{-1}^1 \tilde{P}_{\ell+2m}^{2m}(x) \tilde{P}_n^0(x)\ud x.
\end{equation}
Using the Fourier transform of normalized Legendre polynomials~\cite[\S 7.243 5.]{Gradshteyn-Ryzhik-07}:
\begin{equation}\label{eq:PFT}
\int_{-1}^1 e^{-\ii k x} \tilde{P}_n^0(x) \ud x = 2(-\ii)^n \sqrt{n+\tfrac{1}{2}} j_n(k),
\end{equation}
we may represent the connection coefficients as the composition of a spherical Bessel integral operator and a Fourier integral operator:
\begin{equation}\label{eq:IFTcoefficients}
c_{\ell,n}^{2m} = \dfrac{(-\ii)^n\sqrt{n+\tfrac{1}{2}}}{\pi}\int_\R j_n(k)\ud k \int_{-1}^1 e^{\ii kx} \tilde{P}_{\ell+2m}^{2m}(x) \ud x,
\end{equation}
where the order of integration has been reversed. Eq.~\eqref{eq:IFTcoefficients} clearly shows how a subblock of connection coefficients with indices among $\{\ell\}\times \{n\}$ may be represented as the composition of linear operators involving the indices $n$ and variable $k$, variables $k$ and $x$, and variable $x$ and indices $\ell$. As the rank of the composed operator is bounded by the smallest rank in the composition, the relation to a Fourier integral operator is of paramount importance.

In this integral operator setting, we would need to identify bounds on the relevant wavenumbers to invoke Lemma~\ref{lemma:LowRankFIO}. Before we proceed, we prove a technical lemma.

\begin{lemma}\label{lemma:4F3sum}
For integers $\ell$ and $m$, the summation:
\[
\sum_{j=0}^{[\ell/2]}(2m+1+2\ell-4j)\dfrac{\Gamma(\ell+2m-j+\frac{1}{2})}{\Gamma(\ell+m-j+\frac{3}{2})}\dfrac{(m)_j}{j!}\dfrac{\Gamma(\ell+2m-2j+1)}{\Gamma(\ell-2j+1)} = \dfrac{2\sqrt{\pi}\Gamma(\ell+4m+1)}{16^m\Gamma(m+\frac{1}{2})\Gamma(\ell+1)}.
\]
\end{lemma}
\begin{proof}
Representing $C_\ell^{(2m+\frac{1}{2})}(x)$ in the basis of $C_j^{(m+\frac{1}{2})}(x)$ via~\cite[\S 18.18.16]{Olver-et-al-NIST-10}:
\begin{equation}\label{eq:Cjmh2Cl2mh}
C_\ell^{(2m+\frac{1}{2})}(x) = \sum_{j=0}^{[\ell/2]} \dfrac{2m+1+2\ell-4j}{2m+1}\dfrac{(2m+\frac{1}{2})_{\ell-j}}{(m+\frac{3}{2})_{\ell-j}}\dfrac{(m)_j}{j!}C_{\ell-2j}^{(m+\frac{1}{2})}(x),
\end{equation}
the summation is obtained by using:
\[
C_\ell^{(2m+\frac{1}{2})}(1) = \dfrac{(4m+1)_\ell}{\ell!},\quad{\rm and}\quad \dfrac{\Gamma(2m+\frac{1}{2})\Gamma(2m+2)}{\Gamma(m+\frac{3}{2})\Gamma(4m+1)} = \dfrac{2\sqrt{\pi}}{16^m\Gamma(m+\frac{1}{2})}.
\]
\end{proof}

\begin{theorem}\label{theorem:finitek}
Let:
\[
k_1(n,\varepsilon) := 2\left(\dfrac{\varepsilon}{2}\sqrt{\frac{\pi}{2n+1}}(n+1)\Gamma(n+\tfrac{3}{2})\right)^{\frac{1}{n+1}},
\]
and let:
\[
k_2(\ell,m,n,\varepsilon) := \dfrac{1}{8}\left(\dfrac{2}{\varepsilon}\sqrt{\frac{2n+1}{\pi}}\sqrt{\frac{(2\ell+4m+1)\Gamma(\ell+4m+1)}{\Gamma(\ell+1)}}\dfrac{1}{m\Gamma(m+\frac{1}{2})}\right)^{\frac{1}{m}}.
\]
Then only integration over $k_1(n,\varepsilon) \le \abs{k} \le k_2(\ell,m,n,\varepsilon)$ contributes to Eq.~\eqref{eq:IFTcoefficients} to precision $\varepsilon>0$.
\end{theorem}
\begin{proof}
The absolute value of the connection coefficients is bounded by:
\[
\abs{c_{\ell,n}^{2m}} \le \abs{\dfrac{\sqrt{n+\tfrac{1}{2}}}{\pi}\int_\R j_n(k)\ud k\abs{\int_{-1}^1 e^{\ii kx} \tilde{P}_{\ell+2m}^{2m}(x) \ud x}}.
\]
Using the Cauchy--Schwarz inequality, $|\langle e^{-\ii kx}, \tilde{P}_{\ell+2m}^{2m}\rangle| \le \sqrt{2}$, and thus:
\begin{align*}
\abs{c_{\ell,n}^{2m}} & \le \abs{\dfrac{\sqrt{2n+1}}{\pi}\int_\R j_n(k)\ud k}.
\end{align*}
By the inequality:
\[
\abs{j_n(k)} \le \frac{\sqrt{\pi}}{2\Gamma(n+\frac{3}{2})}\abs{\frac{k}{2}}^n,\quad\forall n\in\N_0,~k\in\R,
\]
we may readily ascertain the equality:
\[
\abs{\dfrac{\sqrt{2n+1}}{\pi}\int_{-k_1}^{k_1} j_n(k)\ud k} = \varepsilon.
\]
Furthermore, an upper bound on $\abs{k}$ may be derived from the fact that the regularity of $\tilde{P}_{\ell+2m}^{2m}(x)\chi_{[-1,1]}(x)$ on the real line dictates the decay rate of its Fourier transform, where $\chi_{[-1,1]}(x)$ is the characteristic function on the interval $[-1,1]$. As $m$ increases, the polynomial decay to $0$ at $\pm1$ from the interior of the interval $[-1,1]$ matches $0$ from the exterior of the interval to $m^{\rm th}$ degree, increasing the regularity proportionally.

In particular, using the Fourier transform~\cite[\S 7.321]{Gradshteyn-Ryzhik-07}:
\[
\int_{-1}^1 e^{\ii kx}C_\ell^{(m+\frac{1}{2})}(x)(1-x^2)^m\ud x = \dfrac{\ii^\ell \sqrt{\pi}\Gamma(\ell+2m+1) j_{\ell+m}(k)}{2^{m-1}\Gamma(\ell+1)\Gamma(m+\frac{1}{2})k^m},
\]
combined with the representation of $C_\ell^{(2m+\frac{1}{2})}(x)$ in the basis of $C_j^{(m+\frac{1}{2})}(x)$ via Eq.~\eqref{eq:Cjmh2Cl2mh}, we obtain the Fourier transform:
\begin{equation}\label{eq:ALPFT}
\int_{-1}^1 e^{\ii kx}P_{\ell+2m}^{2m}(x)\ud x = \dfrac{2^m}{k^m}\sum_{j=0}^{[\ell/2]} (2m+1+2\ell-4j)\dfrac{\Gamma(\ell+2m-j+\frac{1}{2})}{\Gamma(\ell+m-j+\frac{3}{2})}\dfrac{(m)_j}{j!}\dfrac{\ii^{\ell-2j} \Gamma(\ell+2m-2j+1) j_{\ell+m-2j}(k)}{\Gamma(\ell-2j+1)}.
\end{equation}
Combining Lemma~\ref{lemma:4F3sum} with the inequality:
\[
\abs{j_n(k)} \le \sqrt{\frac{1}{\abs{k}}},\quad \forall n\in\N_0,~k\in\R,
\]
we bound the Fourier transform in absolute value by:
\[
\abs{\int_{-1}^1 e^{\ii kx}\tilde{P}_{\ell+2m}^{2m}(x)\ud x} \le \sqrt{\dfrac{(\ell+2m+\frac{1}{2})\Gamma(\ell+4m+1)}{\Gamma(\ell+1)}}\dfrac{2\sqrt{\pi}}{8^m\abs{k}^{m+\frac{1}{2}}\Gamma(m+\frac{1}{2})}.
\]
Then:
\begin{align*}
& \abs{\dfrac{\sqrt{n+\tfrac{1}{2}}}{\pi}\int_{\R\setminus\{\abs{k}<k_2\}} j_n(k)\ud k\abs{\int_{-1}^1 e^{\ii kx} \tilde{P}_{\ell+2m}^{2m}(x) \ud x}},\\
& \le \abs{\dfrac{\sqrt{n+\tfrac{1}{2}}}{\pi}\int_{\R\setminus\{\abs{k}<k_2\}} j_n(k)\ud k \sqrt{\dfrac{(\ell+2m+\frac{1}{2})\Gamma(\ell+4m+1)}{\Gamma(\ell+1)}}\dfrac{2\sqrt{\pi}}{8^m\abs{k}^{m+\frac{1}{2}}\Gamma(m+\frac{1}{2})} },\\
& \le \sqrt{\dfrac{(\ell+2m+\frac{1}{2})\Gamma(\ell+4m+1)}{\Gamma(\ell+1)}}\dfrac{2\sqrt{\pi}}{8^m\Gamma(m+\frac{1}{2})}\abs{\dfrac{\sqrt{n+\tfrac{1}{2}}}{\pi}\int_{\R\setminus\{\abs{k}<k_2\}} \dfrac{\abs{j_n(k)}}{\abs{k}^{m+\frac{1}{2}}}\ud k },\\
& \le \sqrt{\dfrac{(2\ell+4m+1)\Gamma(\ell+4m+1)}{\Gamma(\ell+1)}}\dfrac{1}{8^m\Gamma(m+\frac{1}{2})}\sqrt{\frac{2n+1}{\pi}}\abs{\int_{\R\setminus\{\abs{k}<k_2\}} \frac{{\rm d}k}{\abs{k}^{m+1}} } = \varepsilon.
\end{align*}
\end{proof}

\begin{remark}
\begin{enumerate}
\item Theorem~\ref{theorem:finitek} has a fundamentally different interpretation than~\cite[Theorem~$10$]{ONeil-Woolfe-Rokhlin-28-203-10}, which proves the requisite rank property for Fourier--Bessel transforms with kernel $xJ_\mu(xt)$. With $\mu$ fixed, this theorem is much more closely linked to the Fourier transform due to the oscillatory asymptotics of the Bessel function. In Eq.~\eqref{eq:IFTcoefficients}, the two variables related to indices in the matrices of connection coefficients are $n$ and $\ell$, which do not appear as a product.
\item Estimates for the bounds are $k_1(n,\varepsilon) = \OO(n)$, and for $m=\OO(\ell)$ we find that $k_2(\ell,m,n,\varepsilon) = \OO(\ell)$. The limit $m=\OO(1)$ is known to be compressible due to semiseparability and therefore it is not of interest.
\item It is reasonable to suggest that the estimate for $k_2(\ell,m,n,\varepsilon)$ may be improved upon by using uniform asymptotics of Jacobi polynomials~\cite{Frenzen-Wong-37-979-85,Bai-Zhao-148-1-07} or by refining estimates on the summation in Eq.~\eqref{eq:ALPFT} by incorporating the alternation in sign $\ii^{\ell-2j}$ and spherical Bessel asymptotics. On the other hand, with the current estimates in place, we may conclude by Theorem~\ref{theorem:finitek} that the connection coefficients satisfy {\em a rank property} for successful compression in the butterfly algorithm and we leave the compression to the interpolative decomposition.
\item In fact, $\abs{k} \le k_2(\ell,m,n,\varepsilon)$ is sufficient to prove this rank property but the inclusion of $k_1(n,\varepsilon)\le \abs{k}$ shows that there are certain regions in the $\ell,m,n$-cube where more compression is obtained than others.
\item The rank property revealed by Theorem~\ref{theorem:finitek} is universal in the sense that it is the same low-rank property that appears in the matrices of associated Legendre functions sampled at Gauss--Jacobi quadrature nodes. These matrices are part of Tygert's synthesis and analysis-based approach~\cite{Tygert-229-6181-10} that is also accelerated by the butterfly algorithm.
\end{enumerate}
\end{remark}

The partitioning involved in the butterfly algorithm with six levels is illustrated in the left panel of Figure~\ref{fig:butterflyHODLR}.

\input{tikzHmat}

\begin{figure}[htbp]
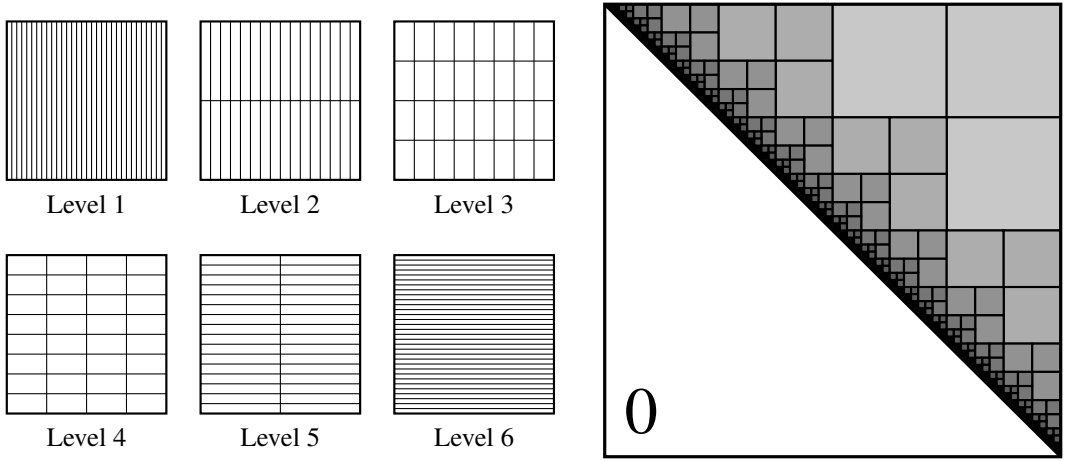

\begin{center}
\begin{tabular}{cc}
\raisebox{0.95\height}{\begin{tabular}{ccc}
\tikz{\butterfly{0}{2.1}{1}{6}} &
\tikz{\butterfly{0}{2.1}{2}{6}} &
\tikz{\butterfly{0}{2.1}{3}{6}}\\
Level $1$ & Level $2$ & Level $3$\\ & & \\
\tikz{\butterfly{0}{2.1}{4}{6}} &
\tikz{\butterfly{0}{2.1}{5}{6}} &
\tikz{\butterfly{0}{2.1}{6}{6}}\\
Level $4$ & Level $5$ & Level $6$\\
\end{tabular}} &
\tikz{\Hmatweak{0}{0}{6}{6}{7}{7}}\\
\end{tabular}
\caption{The spherical harmonic transform proceeds in two steps. Left: an illustration of the butterfly algorithm used to accelerate the conversion of expansions in $\tilde{P}_\ell^{2m}$ to expansions in $\tilde{P}_\ell^0$ and similarly for $\tilde{P}_\ell^{2m+1}$ to $\tilde{P}_\ell^1$. Right: an illustration of the hierarchical decomposition of the matrices of connection coefficients between $\tilde{P}_\ell^0$ and $T_\ell$ and $\tilde{P}_{\ell+1}^1$ and $\sin\theta U_\ell$. The opacity illustrates the data-sparsity in the hierarchical decomposition.}
\label{fig:butterflyHODLR}
\end{center}
\end{figure}

\subsection{An adaptation of the Fast Multipole Method}\label{subsection:FMM}

The right panel of Figure~\ref{fig:butterflyHODLR} describes the familiar hierarchical decomposition of an upper-triangular matrix that characterizes the Fast Multipole Method~\cite{Greengard-Rokhlin-73-325-87}. In this schema, all subblocks in the partition are {\em well-separated} from the main diagonal, relative to their own size. With off-diagonal compression comes accelerated matrix-vector products, among other algebraic properties. Physically, the name is derived from the multipole expansion of an inverse power of the euclidean distance in $\R^n$:
\[
\dfrac{1}{\abs{{\bf r}-{\bf r}_0}^{2\lambda}} = \dfrac{1}{\left(r^2-2rr_0\cos\theta+r_0^2\right)^{\lambda}} = \dfrac{1}{r^{2\lambda}}\sum_{n=0}^\infty \left(\dfrac{r_0}{r}\right)^n C_n^{(\lambda)}(\cos\theta).
\]
In physical space, the assumption of subblocks being well-separated from the main diagonal is equivalent to expansion for sufficiently small $r_0/r$, in which it is readily observed to require $\OO(\log(\varepsilon^{-1}))$ terms in the multipole expansion for approximation to a prescribed accuracy. However, to be precise in this and other settings requires mathematical rigour.

If we refine our hierarchical decomposition of an upper-triangular matrix $A\in\R^{n\times n}$ through $\OO(\log n)$ levels as indices get closer to the diagonal, all subblocks in the partition may be well-approximated by low-rank matrices. In practice, we stop after the dimensions of the subblocks are comparable to the numerical rank required to guarantee an accuracy on the order of machine precision. This implies that the subblocks are no longer data-sparse and nothing is gained from further partitioning.

In the cases of present interest, another asymptotically smooth function succinctly defines the connection coefficients between $\tilde{P}_\ell^0$ and $T_\ell$ and between $\tilde{P}_\ell^1$ and $\sin\theta U_\ell$.

\begin{definition}
The meromorphic function $\Lambda:\C\to\C$ is defined by:
\begin{equation}
\Lambda(z) := \dfrac{\Gamma(z+\frac{1}{2})}{\Gamma(z+1)}.
\end{equation}
\end{definition}

The full asymptotic expansion of $\Lambda(z)$ is derived in~\cite{Elezovic-17-14.2.1-14}, but only a few terms are required in double precision.

\begin{algorithm}[Adapted from Appendix B in~\cite{Bogaert-Michiels-Fostier-34-C83-12}]
If $z > 9.844,\!75$, then 
\begin{align}
\Lambda(z) & \approx \left(1-\dfrac{1}{64(z+\frac{1}{4})^2} + \dfrac{21}{8,\!192(z+\frac{1}{4})^4} - \dfrac{671}{524,\!288(z+\frac{1}{4})^6} + \dfrac{180,\!323}{134,\!217,\!728(z+\frac{1}{4})^8}\right.\nonumber\\
& \qquad \left. - \dfrac{20,\!898,\!423}{8,\!589,\!934,\!592(z+\frac{1}{4})^{10}} + \dfrac{7,\!426,\!362,\!705}{1,\!099,\!511,\!627,\!776(z+\frac{1}{4})^{12}}\right)\left/ \sqrt{z+\tfrac{1}{4}}\right..
\end{align}
gives an approximation to $\Lambda(z)$ that is accurate to the machine precision, $\epsilon_{\rm mach} \approx 2.22\times10^{-16}$. In case $z\le 9.844,\!75$, then we use the recursion:
\begin{equation}
\dfrac{\Lambda(z+1)}{\Lambda(z)} = \dfrac{z+\frac{1}{2}}{z+1},
\end{equation}
until the first condition is satisfied.
\end{algorithm}

Conversion from normalized Legendre polynomials to cosines is given by:
\begin{equation}\label{eq:leg2cheb}
\tilde{P}_n^0(\cos\theta) = \sqrt{n+\tfrac{1}{2}}\sum_{\ell=n,-2}^0 \Lambda(\tfrac{n-\ell}{2})\Lambda(\tfrac{n+\ell}{2})\dfrac{2-\delta_{\ell,0}}{\pi}\cos\ell\theta.
\end{equation}
The inverse relationship is given by:
\begin{equation}
\cos n\theta = -n\sum_{\ell=n,-2}^0 \dfrac{\Lambda(\tfrac{n-\ell-2}{2})\Lambda(\tfrac{n+\ell-1}{2})}{(n-\ell)(n+\ell+1)} \sqrt{\ell+\tfrac{1}{2}}\tilde{P}_\ell^0(\cos\theta).
\end{equation}
Special care must be taken to obtain the limiting values when $n=0$ and when $\ell =n$.

Similarly, conversion from normalized Legendre polynomials of order $1$ to sines is given by:
\begin{equation}
\tilde{P}_{n+1}^1(\cos\theta) = \sqrt{\dfrac{n+\frac{3}{2}}{(n+1)(n+2)}}\sum_{\ell=n,-2}^0 \Lambda(\tfrac{n-\ell}{2})\Lambda(\tfrac{n+\ell+2}{2})\dfrac{2(\ell+1)}{\pi}\sin(\ell+1)\theta,
\end{equation}
and the inverse relationship is given by:
\begin{equation}
\sin (n+1)\theta = -\sum_{\ell=n,-2}^0 \dfrac{(\ell+\frac{3}{2})\Lambda(\tfrac{n-\ell}{2})\Lambda(\tfrac{n+\ell+3}{2})}{(n-\ell-1)(n+\ell+2)}\sqrt{\dfrac{(\ell+1)(\ell+2)}{\ell+\tfrac{3}{2}}}\tilde{P}_{\ell+1}^1(\cos\theta).
\end{equation}

The following definition results in the partitioning in the right panel of Figure~\ref{fig:butterflyHODLR}.

\begin{definition}[Keiner~\cite{Keiner-31-2151-09}]
A square $S\subset \R^2$ defined by the formula $S = [x_0,x_0+c]\times [y_0,y_0+c]$ with $c>0$ is said to be {\em well-separated} if $y_0-x_0\ge2c$.
\end{definition}

Alpert and Rokhlin~\cite{Alpert-Rokhlin-12-158-91} prove rigorous bounds on the univariate Chebyshev interpolants to the $\Lambda$ function on a well-separated square $S$. The bounds are obtained for approximation in either variable while the free variable ranges over all possible values it may take in $S$. In the general ultraspherical setting, Keiner~\cite{Keiner-31-2151-09} shows that the geometric rate of $3^{-k}$ in the original bounds is not optimal. In fact, through analyticity in the open Bernstein ellipse $E_\rho$ with $\rho = 3+\sqrt{8}$, Keiner shows that the optimal geometric decay is $\rho^{-k}$.

Considering Eq.~\eqref{eq:leg2cheb}, the Chebyshev--Legendre transform requires approximation of the function~\cite{Alpert-Rokhlin-12-158-91}:
\[
{\cal M}(x,y) := \dfrac{2}{\pi}\Lambda\left(\frac{y-x}{2}\right)\Lambda\left(\frac{y+x}{2}\right),
\]
on well-separated squares. In the $x$-variable, this results in deriving bounds on the error of the Chebyshev interpolants to:
\[
{\cal M}(x(t),y) = \dfrac{2}{\pi}\Lambda\left(\frac{y-x_0-\frac{c(t+1)}{2}}{2}\right)\Lambda\left(\frac{y+x_0+\frac{c(t+1)}{2}}{2}\right),\quad{\rm for}\quad t\in[-1,1].
\]

\begin{theorem}[Keiner~\cite{Keiner-31-2151-09}]
Let ${\cal M}_k$ denote the degree-$k$ Chebyshev interpolant to ${\cal M}(x(t),y)$. Then for $c > 1$:
\begin{equation}
\sup_{t\in[-1,1]}\abs{{\cal M}(x(t),y)-{\cal M}_k(x(t),y)} \le \dfrac{4M\rho^{-k}}{\rho-1},
\end{equation}
where:
\begin{equation}
M = \dfrac{2\sqrt{2}e^{\frac{5}{3}}}{\pi},\quad{\rm and}\quad \rho = 3+\sqrt{8}.
\end{equation}
\end{theorem}

Thus, a Chebyshev polynomial interpolant of degree $\left\lceil\log_{3+\sqrt{8}}\left(\dfrac{4\sqrt{2}e^{\frac{5}{3}}}{\pi(1+\sqrt{2})\varepsilon}\right)\right\rceil$ approximates ${\cal M}(x(t),y)$ to precision $\varepsilon>0$ on any well-separated square.

In Alpert and Rokhlin's~\cite{Alpert-Rokhlin-12-158-91} Chebyshev--Legendre transform, the Lagrange interpolating formula is originally used to approximate subblocks well-separated from the main diagonal. Since then, Higham~\cite{Higham-24-547-04} has proved the numerical stability of the second barycentric formula, popularized by Berrut and Trefethen~\cite{Berrut-Trefethen-46-501-04}. This formulation also leads to a lower pre-computation since the Chebyshev barycentric weights $\lambda_\ell$ are known analytically.

Let $x_\ell$ and $\lambda_\ell$ be the pair of Chebyshev points and barycentric weights of the first kind:
\[
x_\ell = \cos\left(\dfrac{2\ell+1}{2k+2}\pi\right),\quad{\rm and}\quad\lambda_\ell = (-1)^\ell\sin\left(\dfrac{2\ell+1}{2k+2}\pi\right),\quad{\rm for}\quad \ell = 0,\ldots,k.
\]
Let $a,b,c,d\in\N_0$ and let $m = b-a+1$ and $n = d-c+1$. For each $x\in[a,b]$ and $y\in[c,d]$, consider the degree-$k$ approximation by a polynomial interpolant in the first variable at the mapped Chebyshev points in barycentric form:
\[
f(x,y) \approx p_k(x,y) = \dfrac{\displaystyle \sum_{\ell=0}^k \dfrac{\lambda_\ell f\left(\frac{a+b}{2}+\frac{(b-a)x_\ell}{2},y\right)}{2x-a-b-(b-a)x_\ell}}{\displaystyle \sum_{\ell=0}^k \dfrac{\lambda_\ell}{2x-a-b-(b-a)x_\ell}}.
\]
Consider the matrix that results from sampling the bivariate function $f$ at the integers within the rectangle $[a,b]\times[c,d]$:
\[
[F_{a,b}^{c,d}]_{i,j} := f(i,j)\quad{\rm for}\quad a\le i\le b,~c\le j\le d.
\]
The matrix-vector product $Fv$, where $v\in \R^n$ may be replaced by an approximation that requires only $\OO((k+1)(m+n))$ operations when using $p_k(x,y)$. This is because the barycentric formula allows for the storage of $(k+1)n$ function samples:
\[
[\tilde{F}_{a,b}^{c,d}]_{\ell,j} := f\left(\frac{a+b}{2}+\frac{(b-a)x_\ell}{2},j\right)\in\R^{k\times n},\quad{\rm for}\quad 0\le \ell\le k,~ c\le j\le d,
\]
and $m(k+1)$ weights:
\[
[W_{a,b}]_{i,j} := \left.\dfrac{\lambda_j}{2i-a-b-(b-a)x_j}\right/ \sum_{\ell=0}^k\dfrac{\lambda_\ell}{2i-a-b-(b-a)x_\ell} \in\R^{m\times k},\quad{\rm for}\quad a\le i\le b,~ 0\le j\le k.
\]
With the function samples $\tilde{F}_{a,b}^{c,d}$ and the barycentric weights $W_{a,b}$:
\[
\begin{tikzpicture}[baseline={(current bounding box.center)},scale=1.5,y=-1cm]
  \squarematrix{0.0}{0.0}{$F_{a,b}^{c,d}$}
  \node (approx) at (1.25,0.5) {\Large $\approx$};
  \lowranka{1.5}{0.0}{$W_{a,b}$}
  \lowrankb{2.075}{0.0}{$\tilde{F}_{a,b}^{c,d}$};
\end{tikzpicture}
\]
is the result of interpolating each column of $F_{a,b}^{c,d}$ with a degree-$k$ polynomial.

\begin{remark}
Other methods exist for fast polynomial transforms, such as the Chebyshev--Legendre~\cite{Hale-Townsend-36-A148-14} and Chebyshev--Jacobi transforms~\cite{Slevinsky-17}, and the Toeplitz-dot-Hankel approach~\cite{Townsend-Webb-Olver-17}, all of which are available in {\tt FastTransforms.jl}~\cite{Slevinsky-GitHub-FastTransforms}. After extensive numerical experiments, the hierarchical approach is recommended for this particular problem because the execution is an order of magnitude faster than the pre-computation, rather than the other way around.
\end{remark}

\subsection{Practical methods to reduce the pre-computation}

When a matrix has orthonormal columns, full columns are incompressible in the first step of the butterfly algorithm since the matrix has full rank. Instead of computing IDs of the full columns, the first level of the butterfly algorithm is essentially skipped.

With the ability to convert between neighbouring expansions whose orders differ by two in $\OO(n)$ operations using Givens rotations, the present pre-computation may be ``thinned'' at a modest cost to the execution. Let $k_{\rm avg}$ denote the average numerical rank in the butterfly factorization of any given layer. By \S \ref{subsection:butterfly}, the conversion of one layer of a spherical harmonic expansion to the zeroth or first layer by the butterfly algorithm costs $\OO(k_{\rm avg} n\log n)$; the average numerical rank is a good indicator because subblocks of lower rank have cheaper arithmetic, and for nearly full-rank subblocks the interpolation matrix is nearly a permutation. Similarly, conversion between all spherical harmonic expansions of orders $m+2\mu$, for $\mu = 1,\ldots,k_{\rm avg}$ to those of order $m$ costs $\OO(k_{\rm avg}^2n)$ by Givens rotations. Thus, if the computational complexities are balanced\footnote{And if $k_{\rm avg} = \OO(\log n)$ as supported by the right panel of Figure~\ref{fig:TPMemoryRank}.}, then a ``thin'' plan consists of creating the butterfly factorizations only for orders in stride lengths of $\OO(k_{\rm avg})$. The thin plan, depicted in Figure~\ref{fig:ThinPlan}, has a construction cost of $\OO(1/\log n)$ times the cost of the full pre-computation. In the practical bandlimit of $n < \OO(10,000)$ due to storage limitations on a standard computer, the thin plan's pre-computation is effectively only one order of magnitude larger than execution.

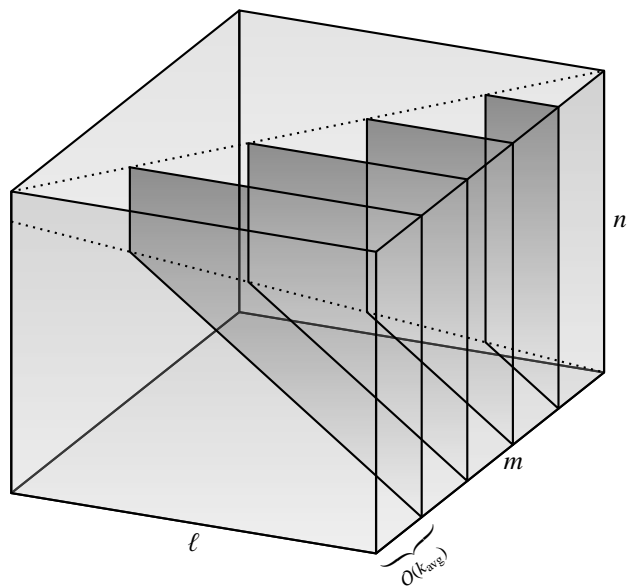
\begin{figure}[htbp]
\begin{center}
\begin{tikzpicture}[x  = {(-0.5cm,-0.4cm)},
                    y  = {(0.6cm,-0.1cm)},
                    z  = {(0cm,1cm)},
                    scale = 2,
                    color = {lightgray}]

\tikzset{facestyle/.style={fill=lightgray,shade,opacity=0.15}}
\tikzset{edgestyle/.style={draw=black,thick,line join=round}}
\tikzset{layerstyle/.style={fill=black,shade}}
\tikzset{layeregdestyle/.style={thick,black,line join=round}}
\tikzset{guidestyle/.style={thick,black,dotted}}

\draw[layerstyle] (3-3/5,0+4/5,2) -- (3-3/5,4,2) -- (3-3/5,4,0) -- (3-3/5,0+4/5,1.8-1.8/5) -- cycle;
\draw[layerstyle] (3-3*2/5,0+4*2/5,2) -- (3-3*2/5,4,2) -- (3-3*2/5,4,0) -- (3-3*2/5,0+4*2/5,1.8-1.8*2/5) -- cycle;
\draw[layerstyle] (3-3*3/5,0+4*3/5,2) -- (3-3*3/5,4,2) -- (3-3*3/5,4,0) -- (3-3*3/5,0+4*3/5,1.8-1.8*3/5) -- cycle;
\draw[layerstyle] (3-3*4/5,0+4*4/5,2) -- (3-3*4/5,4,2) -- (3-3*4/5,4,0) -- (3-3*4/5,0+4*4/5,1.8-1.8*4/5) -- cycle;

\begin{scope}[canvas is yx plane at z=0]
  \path[facestyle] (0,0) rectangle (4,3);
  \path[edgestyle] (0,0) rectangle (4,3);
\end{scope}
\begin{scope}[canvas is zy plane at x=0]
  \path[facestyle] (0,0) rectangle (2,4);
  \path[edgestyle] (0,0) rectangle (2,4);
\end{scope}
\begin{scope}[canvas is zx plane at y=0]
  \path[facestyle] (0,0) rectangle (2,3);
  \path[edgestyle] (0,0) rectangle (2,3);
\end{scope}
\begin{scope}[canvas is zy plane at x=3]
  \path[facestyle] (0,0) rectangle (2,4);
  \path[edgestyle] (0,0) rectangle (2,4);
\end{scope}
\begin{scope}[canvas is zx plane at y=4]
  \path[facestyle] (0,0) rectangle (2,3);
  \path[edgestyle] (0,0) rectangle (2,3);
\end{scope}

\draw[layeregdestyle] (3-3/5,0+4/5,2) -- (3-3/5,4,2) -- (3-3/5,4,0) -- (3-3/5,0+4/5,1.8-1.8/5) -- cycle;
\draw[layeregdestyle] (3-3*2/5,0+4*2/5,2) -- (3-3*2/5,4,2) -- (3-3*2/5,4,0) -- (3-3*2/5,0+4*2/5,1.8-1.8*2/5) -- cycle;
\draw[layeregdestyle] (3-3*3/5,0+4*3/5,2) -- (3-3*3/5,4,2) -- (3-3*3/5,4,0) -- (3-3*3/5,0+4*3/5,1.8-1.8*3/5) -- cycle;
\draw[layeregdestyle] (3-3*4/5,0+4*4/5,2) -- (3-3*4/5,4,2) -- (3-3*4/5,4,0) -- (3-3*4/5,0+4*4/5,1.8-1.8*4/5) -- cycle;

\draw[guidestyle] (3,0,2) -- (0,4,2);
\draw[guidestyle] (3,0,1.8) -- (0,4,0);

\draw[thin,black]
     (3,2,0) node [below] {$\ell$}    
     (1.5,4,0) node [right]        {$\,m$}
     (0,4,1) node [right]        {$n$};
\draw[thin,black]
     (3,4,0) -- node [sloped,below]        {$\underbrace{\quad\quad~~}_{\OO(k_{\rm avg})}$}
     (3-3/5,4,0);
\end{tikzpicture}
\caption{The spherical harmonic transform pre-computation is accelerated by only constructing butterfly factorizations of orders $\OO(k_{\rm avg})$ apart.}
\label{fig:ThinPlan}
\end{center}
\end{figure}

\section{Numerical discussion}

The software package {\tt FastTransforms.jl}~\cite{Slevinsky-GitHub-FastTransforms} written in the {\sc Julia} programming language~\cite{Julia-12,Julia-14} implements the fast and backward stable transforms between spherical harmonic expansions and their bivariate Fourier series. Interpolative decompositions are computed by the {\tt LowRankApprox.jl} package~\cite{LowRankApprox}. The implementations are templated in IEEE single and double precision\footnote{and consequently extended precision, but this hasn't been tested.}, with most matrix-vector multiplications performed by OpenBLAS~\cite{OpenBLAS} at the lowest level. All numerical simulations are performed on a MacBook Pro with a $2.8$ GHz Intel Core i7-4980HQ processor with $4\times256$ KB of L2 cache, $6$ MB of L3 cache, and $16$ GB of $1600$ MHz DDR3 RAM.

Spherical harmonic expansion coefficients $f_\ell^m$ naturally populate a doubly triangular matrix. However, for computational purposes, we organize them into the array:
\[
F = \begin{pmatrix}
f_0^0 & f_1^{-1} & f_1^1 & f_2^{-2} & f_2^2 & \cdots & f_n^{-n} & f_n^n\\
f_1^0 & f_2^{-1} & f_2^1 & f_3^{-2} & f_3^2 & \cdots & 0 & 0\\
\vdots & \vdots & \vdots &  \vdots &  \vdots & \ddots & \vdots & \vdots\\
f_{n-2}^0 & f_{n-1}^{-1} & f_{n-1}^1 & f_n^{-2} & f_n^2 & \iddots & \vdots & \vdots\\
f_{n-1}^0 & f_n^{-1} & f_n^1 & 0 & 0 & \cdots & 0 & 0\\
f_n^0 & 0 & 0 & 0 & 0 & \cdots & 0 & 0\\
\end{pmatrix}\in\R^{(n+1)\times (2n+1)}.
\]
This structure has the advantage of organizing the decay in the coefficients of sufficiently regular functions downward and to the right. The columns of $F$ may be interpreted as having longitudinal basis:
\[
\frac{1}{\sqrt{2\pi}}\{1,e^{-\ii\varphi},e^{\ii\varphi},e^{-\ii2\varphi},e^{\ii2\varphi},\ldots\},\quad{\rm or}\quad \frac{1}{\sqrt{\pi}}\{\tfrac{1}{\sqrt{2}},\sin\varphi,\cos\varphi,\sin2\varphi,\cos2\varphi,\ldots\},
\]
for a real-to-real transform. In the experiments, test spherical harmonic expansion coefficients are first drawn from a standard normal distribution and subsequently, the columns of $F$ are normalized in $\ell^2$. Numerical results report the maximum $\ell^2$ norm over all columns of the error in transforming the spherical harmonic expansion coefficients to Fourier coefficients and back, averaged over three independent trials to reduce the variance. Because the columns of $F$ are normalized, the errors are relative and absolute.

Figure~\ref{fig:SPErrorTimings} reports the numerical results for the $\OO(n^3)$ application of the Givens rotations followed by the rapid conversion from $\tilde{P}_\ell^0$ to $T_\ell$ and $\tilde{P}_\ell^1$ to $\sin\theta U_\ell$. For the sake of comparison, timings are also reported for a 2D discrete cosine transform (DCT) of the same array.

\begin{figure}[htpb]
\begin{center}
\begin{tabular}{cc}
\hspace*{-0.5cm}\includegraphics[width=0.5\textwidth]{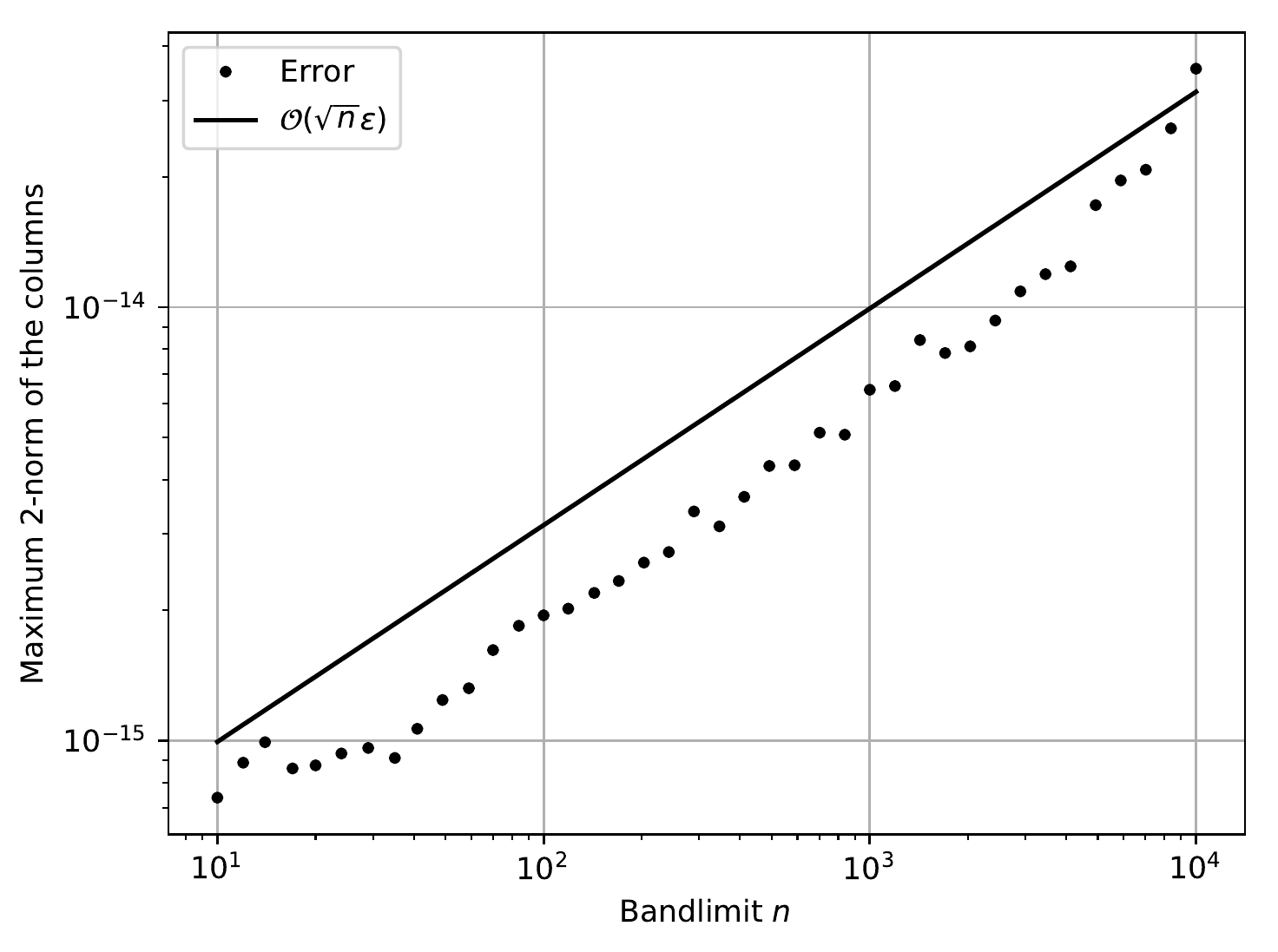}&
\hspace*{-0.5cm}\includegraphics[width=0.5\textwidth]{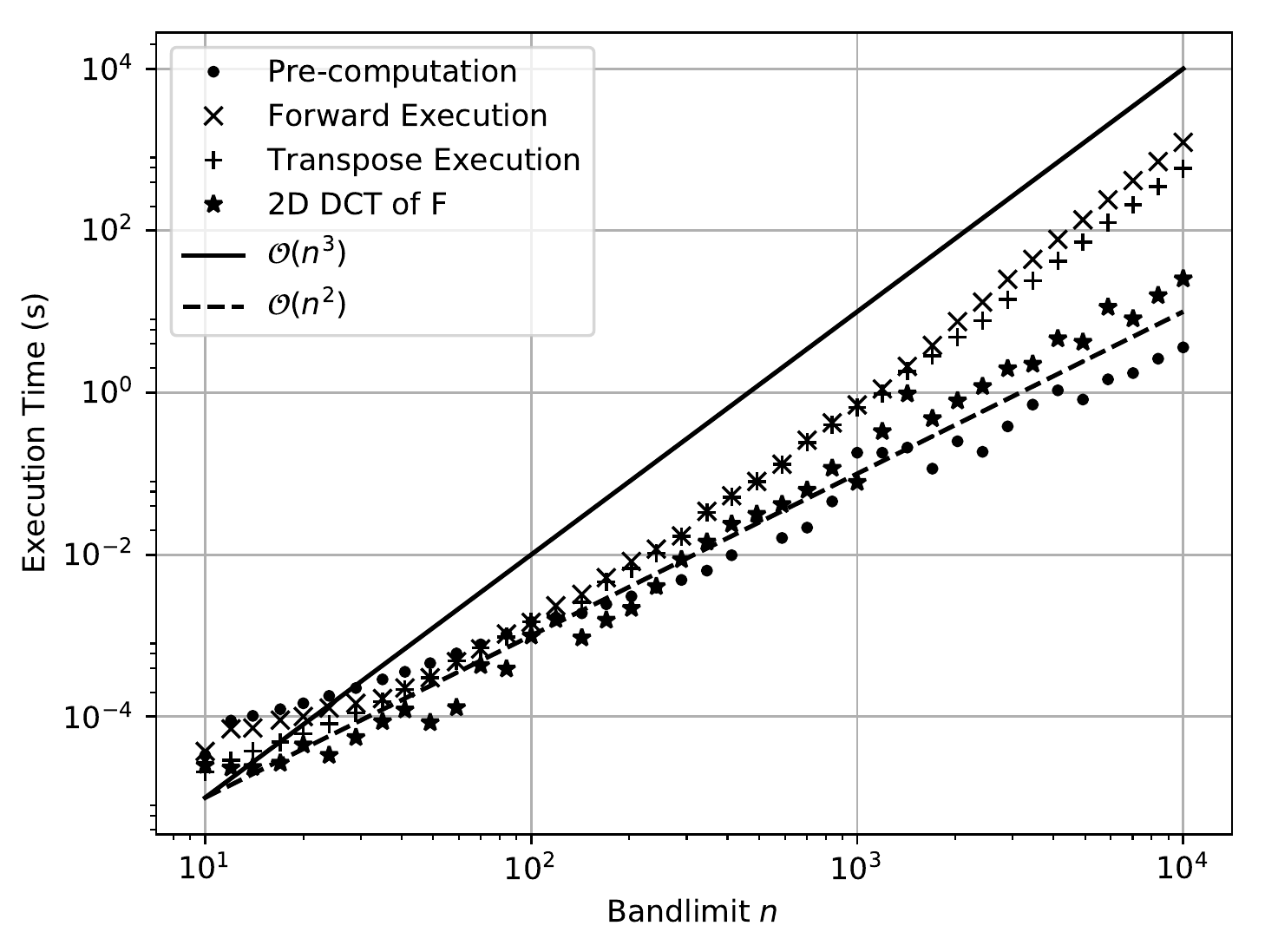}\\
\end{tabular}
\caption{Left: The maximum $2$-norm in the columns of a bandlimited spherical harmonic expansion transformed to its bivariate Fourier series and back. Right: the execution times.}
\label{fig:SPErrorTimings}
\end{center}
\end{figure}

Due to backward stability, the Givens rotations exhibit error growth proportional to $\OO(\sqrt{n}\varepsilon)$. The butterfly factorizations compound this by a factor of $\OO(\sqrt{n})$, as depicted in Figure~\ref{fig:TPErrorTimings} for the thin pre-computation. The thinning parameter of $64$ is utilized in the experiments and not much has been done to further optimize the thinning. Numerical results in the extant literature show error in the numerical evaluation of spherical harmonic expansions at points on the sphere in the $\ell^\infty$ norm. While the $\ell^2$ norm is more natural for the expansion coefficients in the Hilbert space $L^2(\Sph^2,\ud\Omega)$, a very rough conversion would see the present results scaled by $\OO(\log n /\sqrt{n})$.

\begin{figure}[htpb]
\begin{center}
\begin{tabular}{cc}
\hspace*{-0.5cm}\includegraphics[width=0.5\textwidth]{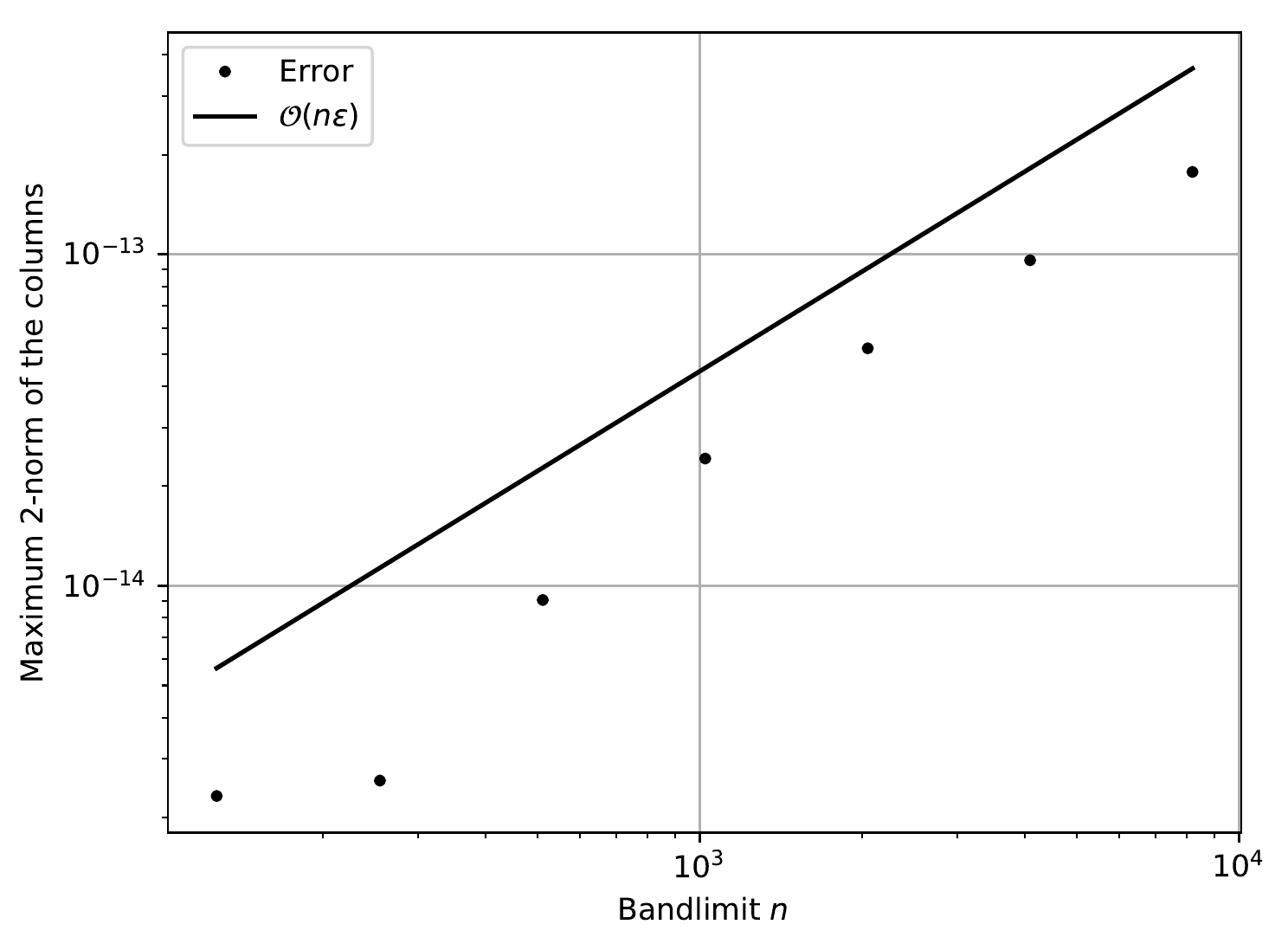}&
\hspace*{-0.5cm}\includegraphics[width=0.5\textwidth]{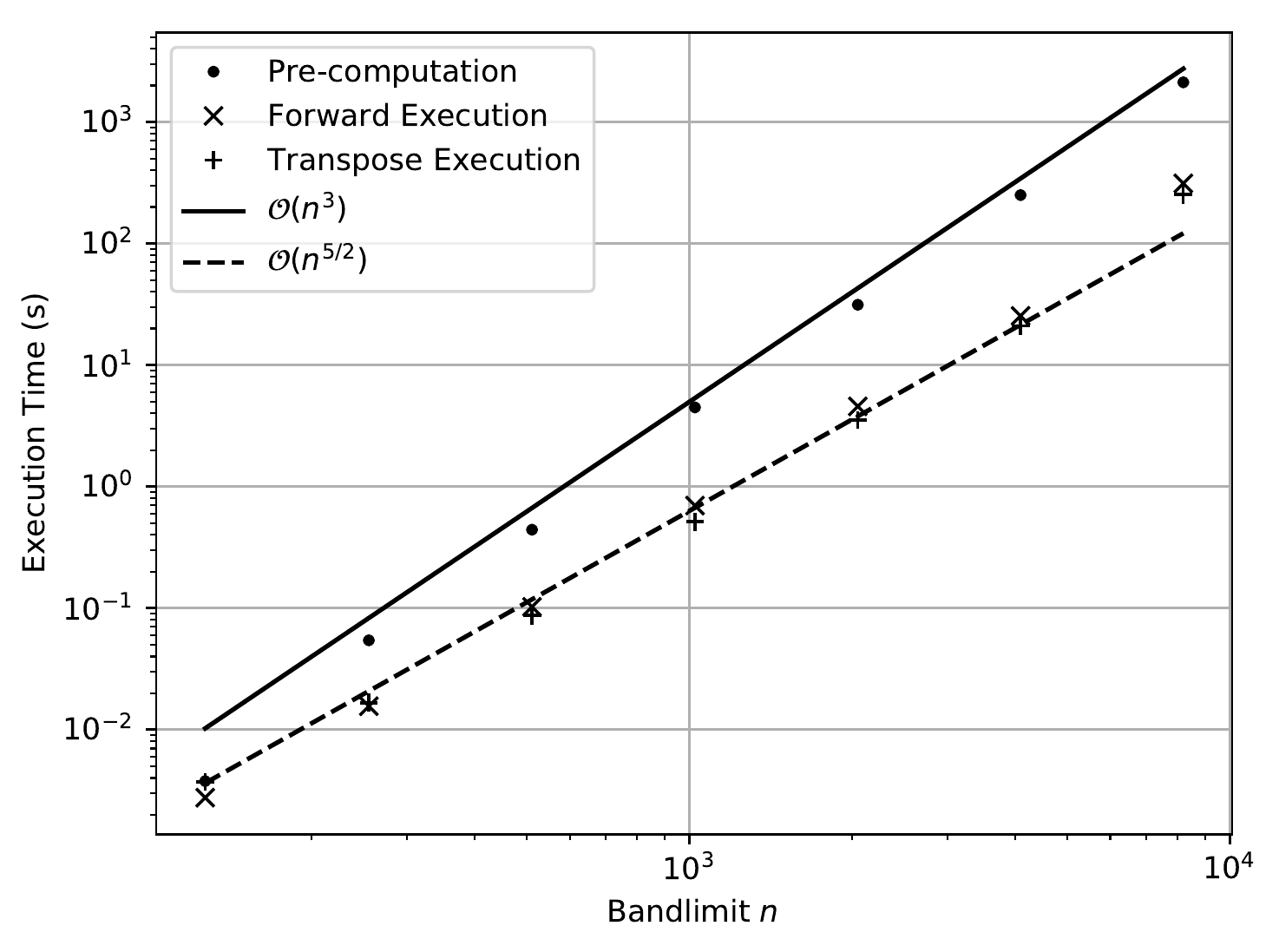}\\
\end{tabular}
\caption{Left: The maximum $2$-norm in the columns of a bandlimited spherical harmonic expansion transformed to its bivariate Fourier series and back. Right: the execution times.}
\label{fig:TPErrorTimings}
\end{center}
\end{figure}

Figure~\ref{fig:TPMemoryRank} shows total memory required to store the thin spherical harmonic pre-computation and the numerical rank statistics. In the right panel of Figure~\ref{fig:TPErrorTimings}, the forward and transpose execution times at a bandlimit of $n=8,191$ are higher than might be extrapolated by the preceding trend. This is explained by the thin pre-computation requiring $26$ GB of memory. Since it requires more memory than may be held in the RAM, the execution times are conflated with memory transfer from the flash hard drive. Additionally, it appears that the execution times follow the $\OO(n^{5/2})$ trend rather than $\OO(n^2\log^2n)$. This is reasonable because the asymptotically optimal complexity of the butterfly algorithm requires a larger bandlimit before becoming fully apparent. This is consistent with Tygert's results~\cite[Tables 1, 3, \& 5]{Tygert-229-6181-10} that begin to show the optimal complexity {\em for a single layer} only for $n\ge\OO(20,000)$.

\begin{figure}[htpb]
\begin{center}
\begin{tabular}{cc}
\hspace*{-0.5cm}\includegraphics[width=0.5\textwidth]{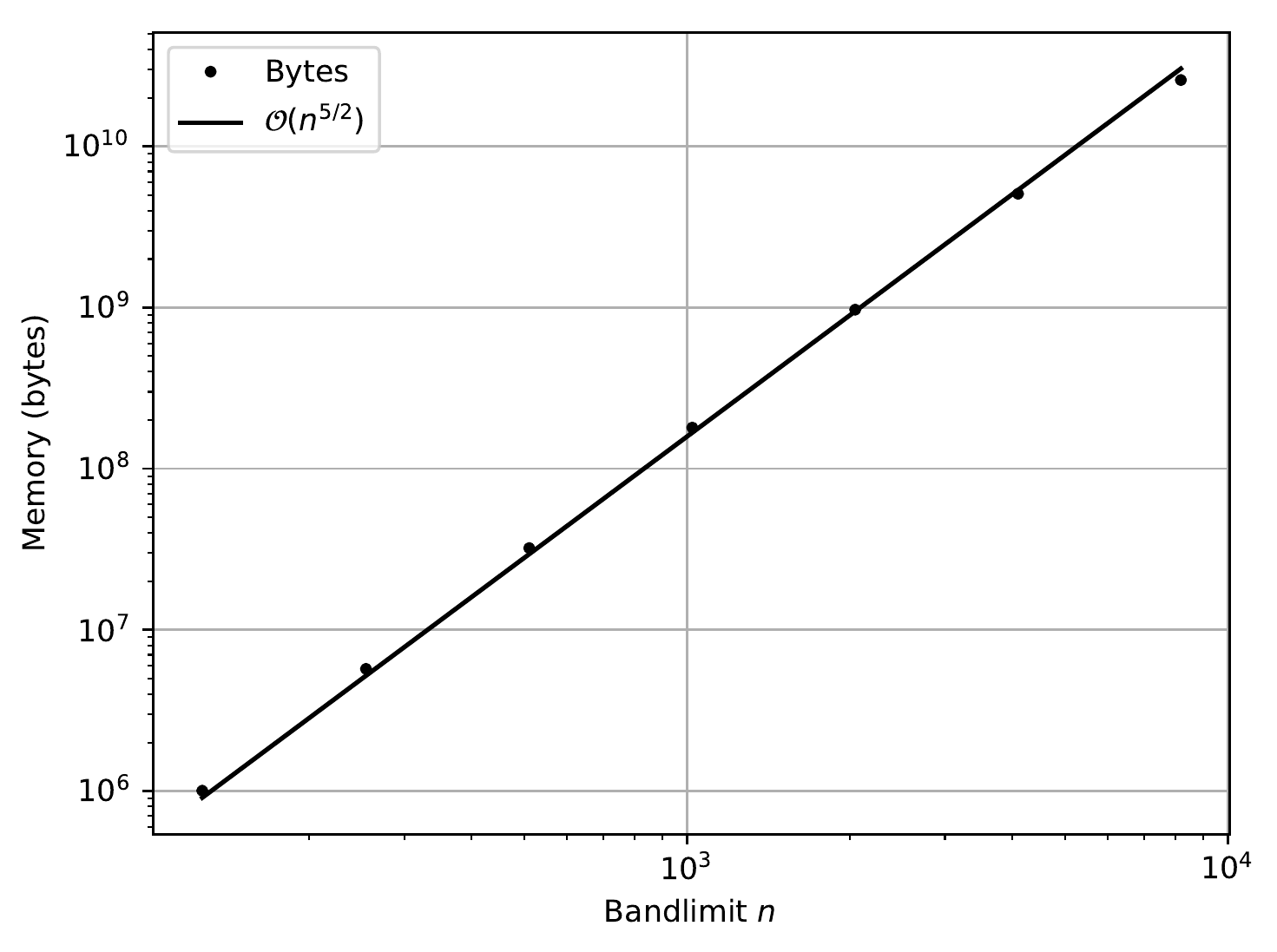}&
\hspace*{-0.5cm}\includegraphics[width=0.5\textwidth]{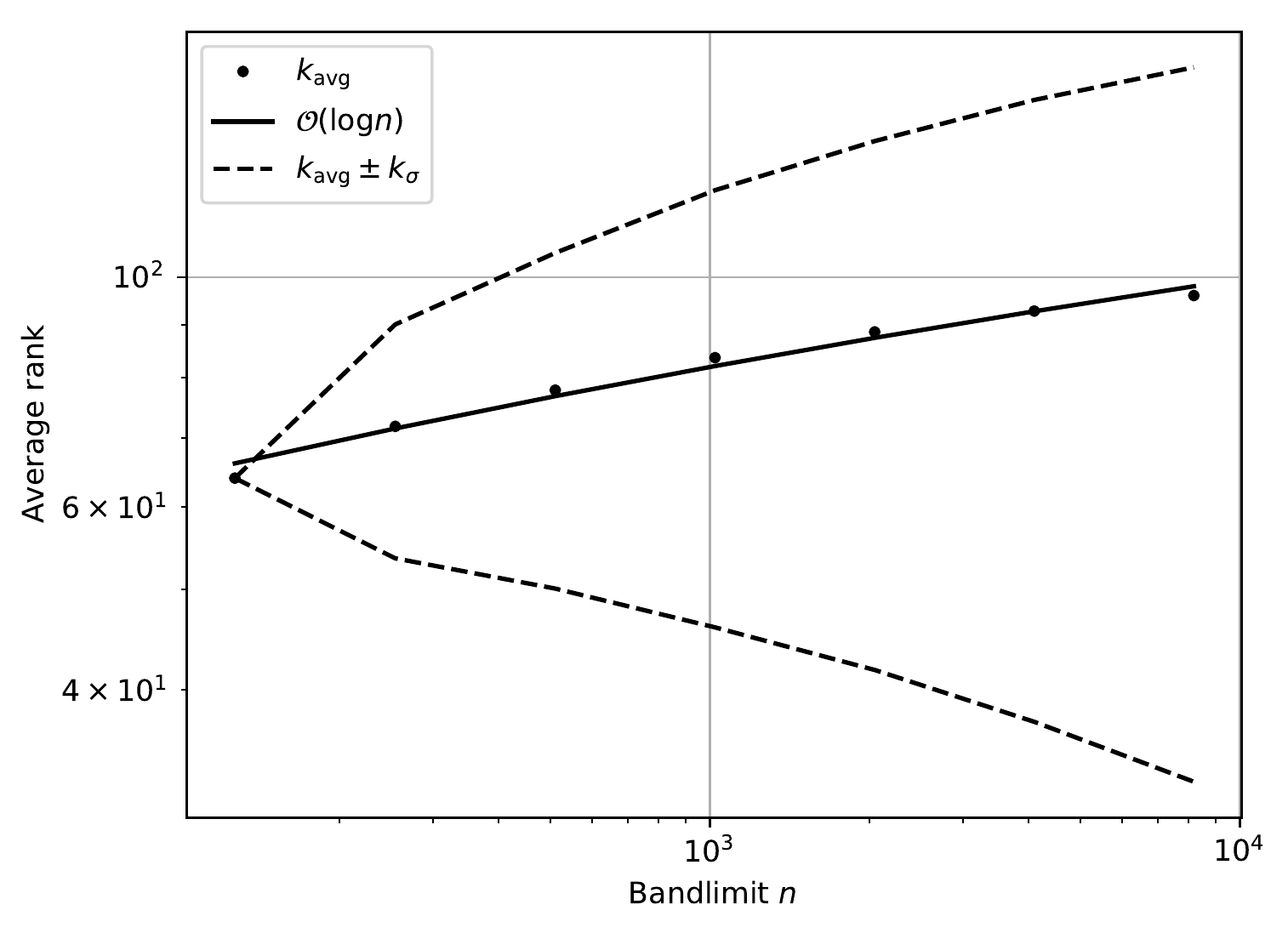}\\
\end{tabular}
\caption{Left: the total number of bytes allocated in the thin spherical harmonic pre-computation. Right: the average numerical rank in the butterfly factorizations bracketed by one standard deviation from the mean.}
\label{fig:TPMemoryRank}
\end{center}
\end{figure}

Tygert has shown~\cite{Tygert-229-6181-10} that the butterfly algorithm may be performed with $\OO(n\log n)$ storage; this is an important consideration for a parallel implementation. For a full spherical harmonic transform, ultimately requiring $\OO(n^2\log n)$ storage, we can afford to perform the butterfly algorithms on a single pair of matrices of connection coefficients that are updated from layer to layer, one each for even and odd layers.

For parallelization over a modest number of processes, the backward stable Givens rotations may still be used: on the master processor, one updates the matrices of connection coefficients and distributes them to the worker processes for compression. For parallelization over numerous processes, application of the Givens rotations would exacerbate asynchrony. Backward stability could be sacrificed for the use of the five-term recurrence relations that are satisfied by the connection coefficients, uncoupling successive layers. However, a parallel implementation is beyond the scope of this report.

\section{Conclusion \& outlook}

The transforms described in the present contribution are rigorously proved to be asymptotically fast and are fast in practice as well. Based on the theory of orthogonal transformations, certain subproblems are well-conditioned and the implementations are backward stable {\em ipso facto} the Givens rotations are known analytically. Thus, this is the first pair of fast transforms to claim backward stability without the use of extended precision in some part of the algorithm or another~\cite{Tygert-227-4260-08,Tygert-229-6181-10,Gruber-Abrykosov-90-525-16}, though some algorithms have empirical evidence substantiating stability. In terms of software, the main contribution is to lower the egregious pre-computation found elsewhere to something more reasonable.

We describe two potential methods to overcome an $\OO(n^3\log n)$ pre-computation. Firstly, Chebyshev polynomial interpolation has been used to accelerate pre-computation in the butterfly algorithm for Fourier integral operators~\cite{Poulson-Demanet-Maxwell-Ying-36-C49-14} and may be further recompressed by another butterfly factorization~\cite{Li-Yang-Martin-Ho-Ying-13-714-15}. In the present context, this requires sampling inner products $\langle \tilde{P}_\nu^0, \tilde{P}_{\lambda+2m}^{2m}\rangle$ and $\langle \tilde{P}_\nu^1, \tilde{P}_{\lambda+2m+1}^{2m+1}\rangle$, where $\nu$ and $\lambda$ are real numbers rather than integers. To this end, the results of Theorem~\ref{theorem:SS} could be extended to larger integer orders and then potentially analytically continued to non-integer degrees. It would take some creativity to obtain $\OO(1)$ computation of entries for real-valued $\lambda$ and $\nu$. Secondly, the Givens rotations may potentially be accelerated on their own. It is surprising in this setting that the analytic structure\footnote{By analytic structure, we mean that the sines and cosines are given analytically by Eq.~\eqref{eq:GRcoefficients}.} of the Givens rotations is known completely. Multiple layers are represented by the sequence of operations:
\[
\begin{tikzpicture}[baseline={(current bounding box.center)},scale=1.5,y=-1cm]
  \foreach \j in {0.0,0.1,0.2} {
    \tikzrotation{2.0*\j+2.4}{\j}
  }
  \foreach \j in {0.6,0.7,0.8} {
    \tikzrotation{2.0*\j+2.2}{\j}
  }
  \node (ddots) at (3.1,0.5) {$\ddots$};

  \foreach \j in {0.0,0.1,0.2} {
    \tikzrotation{2.0*\j+3.4}{\j}
  }
  \foreach \j in {0.6,0.7} {
    \tikzrotation{2.0*\j+3.2}{\j}
  }
  \node (ddots) at (4.1,0.5) {$\ddots$};

  \foreach \j in {0.0,0.1,0.2} {
    \tikzrotation{2.0*\j+4.4}{\j}
  }
  \foreach \j in {0.6} {
    \tikzrotation{2.0*\j+4.2}{\j}
  }
  \node (ddots) at (5.1,0.5) {$\ddots$};
  \node (cdots) at (5.9,0.5) {$\cdots$};
  
\end{tikzpicture}
\]
Since conversions are rectangular, it is possible to trim the number of Givens rotations via turnovers~\cite[Theorem 2.1]{Aurentz-Vandebril-Watkins-35-A255-13}. Turnovers alone do not produce an optimal complexity; however, the overturned structure presents the potential for acceleration. Another possibility is to accumulate Givens rotations into Householder reflectors and relate these reflections in an approximate sense to the Householder decomposition of the DCT.

This work also extends naturally to the conversion of Zernike polynomials~\cite{Vasil-et-al-325-53-16}, orthogonal on the unit disk, to Fourier--Chebyshev series and other families of bivariate analogues of the Jacobi polynomials~\cite{Koornwinder-435-75}, orthogonal on a triangle or a deltoid, among other shapes. These families of bivariate orthogonal polynomials may be organized into layers that are orthogonal functions with respect to the same Hilbert space. This structure ensures that the connection problem is well-conditioned, and allows us to formulate backward stable algorithms.

\section*{Acknowledgments}

This work is motivated by the earnest encouragement of Sheehan Olver, Alex Townsend, Geoff Vasil, and Marcus Webb. I thank Jared Aurentz and Behnam Hashemi for discussions on backward stable numerical linear algebra. I acknowledge the generous support of the Natural Sciences and Engineering Research Council of Canada through a postdoctoral fellowship (6790-454127-2014) whence this work began and a discovery grant (RGPIN-2017-05514).

\bibliography{/Users/Mikael/Bibliography/Mik}

\end{document}

%% file: header_els.tex
\usepackage[T1]{fontenc}
\usepackage{times}
\usepackage{epsfig,graphicx,tabularx,color}
\usepackage[cp850]{inputenc}
\usepackage[english]{babel}
\usepackage{amsmath,amssymb,amsfonts,amsthm,mathrsfs,comment,mathdots,multirow,tikz}
\usepackage{rotating, longtable, relsize}
\usepackage[f]{esvect}
\usepackage{times}
\usepackage{fancyhdr}
\usepackage{lipsum}

\usetikzlibrary{calc,3d}

\newtheorem{theorem}{Theorem}[section]

\newtheorem{lemma}[theorem]{Lemma}

\newtheorem{algorithm}[theorem]{Algorithm}
\newtheorem{definition}[theorem]{Definition}

\theoremstyle{definition}
\newtheorem{remark}[theorem]{Remark}

\def\XXint#1#2#3{{\setbox0=\hbox{$#1{#2#3}{\int}$}
     \vcenter{\hbox{$#2#3$}}\kern-.5\wd0}}

\def\ud{{\rm\,d}}
\def\fl{{\rm\,fl}}

\def\C{\mathbb{C}}

\def\N{\mathbb{N}}

\def\R{\mathbb{R}}
\def\Sph{\mathbb{S}}

\def\OO{\mathcal{O}}

\def\pr(#1){\left({#1}\right)}
\def\br[#1]{\left[{#1}\right]}

\def\abs#1{\left|{#1}\right|}
\def\norm#1{\left\|{#1}\right\|}
\def\conj#1{\overline{#1}}

\def\ii{{\rm i}}
\def\for{\hbox{ for }}

\newcommand{\rank}{\operatorname{rank}}

%% file: tikzHmat.tex
\newcount\nlevels

\def\Hmatstrong#1#2#3#4#5#6{
\nlevels=#5
\ifnum\the\nlevels>1\relax
  \advance\nlevels by-1\relax
  
  \pgfmathparse{0.5*#1+0.5*#3}\edef\tempx{\pgfmathresult}
  \pgfmathparse{0.5*#2+0.5*#4}\edef\tempy{\pgfmathresult}

  \pgfmathparse{85*(1-\the\nlevels/#6)}\edef\opacity{\pgfmathresult}
  \filldraw[thick, color = black!100, fill = black!\opacity] (\tempx, \tempy) rectangle (#3, #4);

  \begingroup
    \edef\Hmatstrongone{\noexpand\Hmatstrong{#1}{\tempy}{\tempx}{#4}{\the\nlevels}{#6}}
    \edef\Hmatstrongtwo{\noexpand\Hmatstrong{\tempx}{#2}{#3}{\tempy}{\the\nlevels}{#6}}
    \Hmatstrongone\Hmatstrongtwo
  \endgroup

  \advance\nlevels by1\relax
\else
  \fill[black] (#1, #4) -- (#3, #2) -- (#3, #4) -- cycle;
\fi
\ifnum\the\nlevels=#6
  \draw[very thick, black] (#1,#2) rectangle (#3,#4);
  \draw[very thick, black] (#3,#2) -- (#1,#4);
  \node (zero) at (0.92*#1+0.08*#3,0.9*#2+0.1*#4) {\Huge $0$};
\fi
}

\def\Hmatweak#1#2#3#4#5#6{
\nlevels=#5
\ifnum\the\nlevels>1\relax
  \advance\nlevels by-1\relax
  
  \pgfmathparse{0.5*#1+0.5*#3}\edef\tempx{\pgfmathresult}
  \pgfmathparse{0.5*#2+0.5*#4}\edef\tempy{\pgfmathresult}

  \begingroup
    \edef\Hmatweakone{\noexpand\Hmatweak{#1}{\tempy}{\tempx}{#4}{\the\nlevels}{#6}}
    \edef\Hmatweaktwo{\noexpand\Hmatweakprime{\tempx}{\tempy}{#3}{#4}{\the\nlevels}{#6}}
    \edef\Hmatweakthree{\noexpand\Hmatweak{\tempx}{#2}{#3}{\tempy}{\the\nlevels}{#6}}
    \Hmatweakone\Hmatweaktwo\Hmatweakthree
  \endgroup

  \advance\nlevels by1\relax
\else
  \fill[black] (#1, #4) -- (#3, #2) -- (#3, #4) -- cycle;
\fi
\ifnum\the\nlevels=#6
  \draw[very thick, black] (#1,#2) rectangle (#3,#4);
  \draw[very thick, black] (#3,#2) -- (#1,#4);
  \node (zero) at (0.92*#1+0.08*#3,0.9*#2+0.1*#4) {\Huge $0$};
\fi
}

\def\Hmatweakprime#1#2#3#4#5#6{
\nlevels=#5
\ifnum\the\nlevels>1\relax
  \advance\nlevels by-1\relax
  
  \pgfmathparse{0.5*#1+0.5*#3}\edef\tempx{\pgfmathresult}
  \pgfmathparse{0.5*#2+0.5*#4}\edef\tempy{\pgfmathresult}

  \pgfmathparse{75*(1-\the\nlevels/#6)}\edef\opacity{\pgfmathresult}
  
  \filldraw[thick, color = black!100, fill = black!\opacity] (#1, \tempy) rectangle (\tempx, #4);
  \filldraw[thick, color = black!100, fill = black!\opacity] (\tempx, \tempy) rectangle (#3, #4);
  \filldraw[thick, color = black!100, fill = black!\opacity] (\tempx, #2) rectangle (#3, \tempy);

  \begingroup
    \edef\Hmatweakone{\noexpand\Hmatweakprime{#1}{#2}{\tempx}{\tempy}{\the\nlevels}{#6}}
    \Hmatweakone
  \endgroup

  \advance\nlevels by1\relax
\else
  \filldraw[thin, black] (#1, #2) rectangle (#3, #4);
\fi
}

\def\butterfly#1#2#3#4{
  \pgfmathparse{(#2-#1)/2.0^(#4-#3)}\edef\widthx{\pgfmathresult}
  \pgfmathparse{(#2-#1)/2.0^(#3-1)}\edef\widthy{\pgfmathresult}

  \draw[thick, color = black] (#1, #1) rectangle (#2, #2);

  \foreach \i in {0.0,\widthx,...,#2} {
    \draw[thin, color = black] (#1 + \i, #1) -- (#1 + \i, #2);
  }
  \foreach \j in {0.0,\widthy,...,#2} {
    \draw[thin, color = black] (#1, #1 + \j) -- (#2, #1 + \j);
  }
}